\documentclass{amsart}

\usepackage{amsfonts,amsmath,amsthm}	
\usepackage{amssymb}

\usepackage{graphics}

\newtheorem{theorem}{Theorem}[section]
\newtheorem{proposition}[theorem]{Proposition}
\newtheorem{corollary}[theorem]{Corollary}
\newtheorem{lemma}[theorem]{Lemma}

\theoremstyle{definition}
\newtheorem{definition}[theorem]{Definition}
\newtheorem{example}[theorem]{Example}

\theoremstyle{remark}
\newtheorem{remark}[theorem]{Remark}

\numberwithin{equation}{section}

\newcommand{\dist}{\ensuremath{\mathrm{dist} }}
\newcommand{\grad}{\ensuremath{\mathrm{grad}\ }}

\newcommand{\tub}{\ensuremath{\mathrm{Tub} }}
\newcommand{\F}{\ensuremath{\mathcal{F}}}
\newcommand{\singularF}{\ensuremath{\mathcal{X}_{F}}}

\newcommand{\codim}{\ensuremath{\mathrm{codim}\ }}

\newcommand{\id}{\ensuremath{\mathrm{id}}}

\newcommand{\RR}{\mathbb R}

\newcommand{\metric}{\ensuremath{ \mathrm{g} }}

\newcommand{\ip}[2]{\ensuremath{\langle{#1},{#2}\rangle}}

\newcommand{\minimalstratum}{\ensuremath{ \mathcal{M} }}

\begin{document}

\title{Progress in  the Theory of Singular Riemannian Foliations}

\author{Marcos M. Alexandrino}

\author{Rafael Briquet }

\author{Dirk T\"{o}ben}

\address{Marcos M. Alexandrino (corresponding author),Rafael Briquet, Dirk T\"{o}ben \hfill\break\indent
Instituto de Matem\'{a}tica e Estat\'{\i}stica\\
Universidade de S\~{a}o Paulo, \hfill\break\indent
 Rua do Mat\~{a}o 1010,05508 090 S\~{a}o Paulo, Brazil}
\email{marcosmalex@yahoo.de, malex@ime.usp.br}




\thanks{The first author was supported by a
research productivity scholarship  from CNPq-Brazil
 and partially supported by FAPESP (S\~{a}o Paulo, Brazil).  The second author contribution was written at IME-USP
while researching for his MSc \cite{Rafael} and was
supported by a  research scholarship from CNPq. The third author was also supported by a scholarship from CNPq-Brazil. }




\begin{abstract}
A singular foliation is called a singular Riemannian foliation (SRF) if every geodesic that is perpendicular to one leaf is perpendicular to every leaf it meets.
A typical example is the partition of a complete Riemannian manifold into orbits of an isometric action.

In this survey, we provide an introduction to the theory of SRFs, leading from the foundations to recent developments in research on this subject. Sketches of proofs are included and useful techniques are emphasized.
We study the local structure of SRFs in general and under curvature conditions. We review the solution of the Palais-Terng problem on integrability of the horizontal distribution. Important special classes of SRFs, like polar and variationally complete foliations and their connections, are treated. A characterisation of SRFs whose leaf space is an orbifold is given. Moreover, desingularizations of SRFs are studied and applications, e.g., to Molino's conjecture, are presented.

\end{abstract}


\maketitle

\tableofcontents

\section{Introduction}

A partition $\F$ of a complete Riemannian manifold $M$ into connected immersed submanifolds (the \emph{leaves}) is called a {\it singular Riemannian foliation} (SRF for short) if it satisfies the following two conditions:

\begin{enumerate}
\item $\F$ is a {\it singular foliation},
i.e., the module $\singularF$ of smooth vector fields on $M$ that are tangential at each point to the corresponding leaf acts transitively on each leaf. In other words, for each leaf $L$ and each $v\in TL$ with footpoint $p,$ there is $X\in \singularF$ with $X(p)=v$.
\item   Every geodesic that is perpendicular at one point to a leaf is \emph{horizontal}, i.e.,  is perpendicular to every leaf it meets.
\end{enumerate}

A leaf $L$ of $\F$ (and each point in $L$) is called \emph{regular} if the dimension of $L$ is maximal, otherwise $L$ is called {\it singular}.

SRFs were defined by Molino \cite{Molino} in his study of Riemannian foliations. His main example in mind were the decomposition of the ambient manifold by the leaf closures of a Riemannian foliation. But he also noticed that the orbit decomposition of an isometric actions of a Lie group provides an example. It was later realized that the notion is also suited to describe the set of parallel submanifolds of an isoparametric submanifold. This generality of the notion of a SRF together with the simplicity of its definition makes a SRF an attractive and worthwhile object to study.

The intention of this paper is to give an introduction to singular Riemannian foliations and to present recent  results.

In Section \ref{sec:local} we want to review  some local structural theorems for SRFs. For instance, we prove in Theorem \ref{lemma-metric-in-S} that the restriction of a SRF to a slice of a leaf lifted to its normal space via the exponential map is again a SRF (the latter is called the \emph{infinitesimal foliation}). This reduces the local study of a SRF to Euclidean space just as the slice representation in the theory of proper Lie group actions.

Another important tool is presented in Section \ref{sec:Bottequifocality}. We know that for an isometric Lie group action on a complete connected manifold we can recover the whole orbit decomposition by the $G$-equivariance of the normal exponential map of a principle orbit $G\cdot p$, more precisely, $G \cdot\exp\xi=\exp(G\xi)$ for each normal vector $\xi$ of $G\cdot p$. In Theorem \ref{thm-s.r.f.-equifocal} we prove an analogue for SRFs which we call \emph{Bott-equifocality}.

A rich class of SRFs is provided by polar actions, i.e., isometric Lie group actions such that through every point there is a submanifold, called a {\em section}, that meets every orbit and always orthogonally and in the regular orbits even transversally. This is equivalent to the integrability of the normal distribution on the regular stratum by Theorem \ref{teo-integra-polar}. A typical example of a polar action is the action of a compact Lie group on itself by conjugation; the maximal tori are the sections. More generally the isotropy actions on symmetric spaces are polar with maximal flats as sections. A very similar structure can be found in submanifold theory. The decomposition by parallel submanifolds of an isoparametric submanifold in the sphere or in Euclidean space and of their generalization, an equifocal submanifold in a symmetric space of compact type, exhibit similar properties, see e.g.,  Thorbergson \cite{Th,ThSurvey2,ThSurvey3};   Berndt, Console and Olmos \cite{BerndtConsoleOlmos} and  Terng and Palais \cite{PalaisTerng}. Now both these classes can be subsumed into a more general notion, namely, into SRFs admitting sections, nowadays called polar SRFs, where the definition of a section is adapted in the obvious way. In fact, some of the properties that are shared by polar actions and isoparametric foliations can be derived for polar foliations.
For a polar action the slice theorem can be strengthened. It says that the isotropy action on a slice of an orbit is again a polar action. Likewise the slice theorem for SRFs can be strengthened to say that the restriction of a polar foliation to a slice is again a polar foliation, see Theorem \ref{teo-slice-polar}.
These and further results are explored in Section \ref{sec:polar}.

Another class of SRFs are variationally complete foliations (VC for short).
Variational completeness for a SRF/action is the absence of conjugate points in the leaf/orbit space.
The relations between polar and VC actions have been investigated by different authors as Bott, Samelson, Conlon and others. For instance,
hyperpolar foliations, i.e., polar actions with flat sections, are VC. The converse is true if the ambient manifold has nonnegative curvature. This result which is due to Lytchak and Thorbergsson is reviewed in Section \ref{sec:VC}.

Recently Lytchak and Thorbergsson defined a SRF to be \emph{infinitesimal polar} if each of its  infinitesimal foliations is polar. It follows from the slice theorem (Theorem \ref{teo-slice-polar}) mentioned above that polar foliations are infinitesimal polar. Due to Lytchak and Thorbergson VC foliations are also infinitesimally polar, see Theorem \ref{teovc}.
As we will see in Section \ref{sec:infpolar}, infinitesimally polar foliations are characterized as those SRFs whose local leaf space is an orbifold. Another equivalent condition is that the curvature of the leaf space  does not explode, see Theorem \ref{thm-infpol}.

In Section \ref{sec:desingularization} we discuss the desingularization of SRFs. We say that a SRF $\F$ on $M$ admits a \emph{desingularization} if there exists a foliated map from a (desingularized) Riemannian foliation to $(M,\F)$.
When the SRF is infinitesimally polar there is a desingularization that preserves the transverse length metric, see Theorem \ref{thm:desing}. On the other hand, if for a SRF there is a desingularization that preserves the transverse length metric, the leaf space of the SRF coincides with the leaf space of the desingularized  Riemannian foliation, and is therefore an orbifold. Thus the SRF is infinitesimally polar by the above characterization.
In general the leaf space of a SRF $\F$ is not an orbifold; in fact we have already discussed restrictions for a SRF to have an orbifold as leaf space. However, $M/\F$ can be arbitrarily  well approximated in the sense of Gromov-Hausdorff by a Riemannian orbifold, that is the leaf space of a desingularized  Riemannian foliation; see Corollary \ref{cor-gromov-hausdorff-convergence}.

We  present a selection of recent results on SRFs that are very related to each other. In paticular, many proof sketches presented here use
changes of metrics or Bott-equifocality presented respectively in Section 2 and 3.
 In Section \ref{sec:question} we will summarize a few  other important results and we finish this paper with open questions and perspectives in this field.

In this presentation we tried to balance brevity of explanations with the amount of prerequisites required by the reader. In particular we did not follow the chronological order of the appearance of results; a historical account can be found in surveys by Thorbergsson \cite{Th,ThSurvey2,ThSurvey3}). Since this is a survey that also includes technical aspects it is suited for researchers who want to begin working in this field.


\section{The local structure of SRFs}\label{sec:local}
In this section we present the basic tools that will be used throughout this paper.
The main results are Theorem \ref{lemma-almost-product} and \ref{lemma-metric-in-S}. These are local structure results allowing us to reduce the local study of a SRF to a Euclidean space just as the slice representation in the theory of proper Lie group actions.

Before we present these results let us fix the notation.

We will assume that $\F$ is a SRF on a complete Riemannian manifold $M.$

Let $p\in M$ and $P$ be a relatively compact connected open neighborhood of $p$ in the leaf $L_p$ through $p$. We write $\tub_r(P)$ for the tubular neighborhood of $P$ of radius $r$. There is an $\varepsilon>0$ such that the normal exponential map $\exp^\nu:\{v\in\nu P\mid \|v\|<\varepsilon\}\to \tub_\varepsilon(P)$ from the space of normal vectors of length $<\varepsilon$ to the tubular neighborhood of $P$ of radius $\varepsilon$ is a diffeomorphism. Then for $0<\lambda\leq 1$ the homothety $h_\lambda:\tub_\varepsilon(P)\to\tub_{\lambda\varepsilon}(P)$ is defined by $h_\lambda(\exp^\nu(v))=\exp^\nu(\lambda v)$.

We write $\pi:\tub_\varepsilon(P)\to P$ for the orthogonal projection and for $q\in P$ we call the fiber $S_q:=\pi^{-1}(q)$ the {\em slice} through $q$. For $y\in \tub_\varepsilon(P)$ we call the connected component $P_y$ of $L_y\cap\tub_\varepsilon(P)$ containing $y$ the {\em plaque} of $\F$ through $y$ in the neighborhood.

We also need to recall the so-called \emph{Homothetic Transformation Lemma} of Molino (see Lemma 6.2 \cite{Molino}).
\begin{proposition}[\cite{Molino}]
\label{homothetic-lemma}
The homothetic transformation $h_{\lambda}$ sends plaque to plaque and therefore respects the singular foliation $\F$ in the tubular neighborhood $\tub(P)$.
\end{proposition}

The next two theorems are the main results of this section and contain some  improvements of  Molino's results (compare with Theorem 6.1 and Proposition 6.5 of \cite{Molino}). In particular Theorem \ref{lemma-metric-in-S} can be seen as a generalization of the slice representation of proper actions.


\begin{theorem}[\cite{AlexToeben2,Molino}]
\label{lemma-almost-product}
\
Let $g$ be the original metric on $M$ and $q\in M.$  Then there exists a tubular neighborhood $\tub(P_{q})$ and  a new metric $\tilde{g}$ on $\tub(P_{q})$ with the following properties:
\begin{enumerate}
\item[(a)] For each $x\in \tub(P_{q})$ the normal space of the leaf $L_x$ is tangential to the slice $S_{\tilde{q}}$ containing $x$, where $\tilde q\in P_q.$
\item[(b)] The restriction $\pi|_{P_{x}}$ of the orthogonal projection $\pi:\tub(P_{q})\rightarrow P_q$ is a Riemannian submersion.
\item[(c)] $\F| \tub(P_{q})$ is a SRF.
\item[(d)] $\F| S_{\tilde{q}}$ is a SRF for each $\tilde{q}\in P_{q}.$
\item[(e)] The associated transverse metric is not changed, i.e., the distance between the plaques with respect to $g$ is the same as with respect to $\tilde{g}$.
\item[(f)] The geodesics orthogonal to $P_{q}$ are the same for $g$ and $\tilde{g}$.
\end{enumerate}
\end{theorem}
\begin{proof}

Let $X_1,\ldots, X_r \in  \singularF$ (i.e. vector fields that are always tangential to the leaves)
so that $\{X_i(q)\}_{i=1,\ldots ,r}$ is a linear basis of $T_{q}P_{q}$.  Let $\varphi_{t_1}^1,\ldots, \varphi_{t_r}^r$ denote the associated one parameter groups and define $\varphi(t_1,\ldots,t_r,y):= \varphi_{t_1}^1\circ\cdots\circ\varphi_{t_r}^r(y)$ where $y\in S_{q}$ and $(t_1,\ldots,t_r)$ belongs to a neighborhood $U$  of $0\in\mathbb{R}^r.$  Then, reducing $U$ and $\tub(P_{q})$ if necessary, one can guarantee the existence of   a regular foliation $\F^{2}$ with plaques   $P^{2}_{y}= \varphi(U,y).$ We note that the plaques $P^{2}_{z}\subset P_{z}$ and each plaque $P^{2}$ intersects each slice transversally and at exactly one point.  Using the fact that  $\pi|_{P^{2}_{y}}: P^{2}_{y}\rightarrow P_{q}$ is a diffeomorphism, we can define a metric on each plaque $P^{2}_{y}$ as
$\tilde{g}^{2}:=(\pi|_{P^{2}_{y}})^{*}g.$

Now we want to define a metric $\tilde{g}^1$ on each slice $S\in \{S_{\tilde{q}}\}_{\tilde{q}\in P_{q}} .$ Set $D_{p}:=\nu_{p} L^{2}_{p}$ and define $\Pi:T_{p}M\rightarrow D_{p}$ as the orthogonal projection with respect to $g$. The fact that each plaque $P^{2}$ intersects each slice transversally and at a single point implies that $\Pi|_{T_{p}S}:T_{p}S\rightarrow D_{p}$ is an isomorphism. Finally we define $\tilde{g}^{1}:= (\Pi|_{T_{p}S})^{*}g$ and $\tilde{g}:=\tilde{g}^{1}+\tilde{g}^{2}$, meaning that $\F^2$ and the slices meet orthogonally. Items (a) and (b) follow directly from the definition of $\tilde{g}.$

Before we continue the sketch of proof, let us first recall the definition of a transnormal function. A smooth function $f:M\rightarrow \mathbb{R}$ is called a \emph{transnormal function} with respect to the metric $g$ if  there exists a $C^{2}(f(M))$ function  $b$  such that $g(\grad f,\grad f)=b\circ f$.
According to Wang \cite{Wang} there are at most two critical level sets of the transnormal function $f$ and each regular level set of $f$ is a distance cylinder over them.

For any $x\in \tub(P_q)$ let $f:\tub(P_{x})\rightarrow \mathbb{R}$ be a smooth transnormal function with respect to the metric $g$ so that each regular level set $f^{-1}(c)$ is a cylinder $C$ with axis $P_x$,  e.g. $f(y)=d(y,P_x)^2$. It is possible to prove that $f$ remains a transnormal function with respect to the new metric $\tilde{g}$ (with the same function $b$). Using a local version of Q-M Wang's theorem, we conclude that each regular level set $C$ of $f$ is a distance cylinder around $P_{x}$ with respect to the metric $\tilde{g}$. Therefore plaques of $\F$ are locally equidistant to each other, i.e., we have proved (c).
Item (d) follows then from (a).
To prove (e) we have to prove that the distance between the cylinder $C$ and the plaque $P_{x}$ is the same for both metrics.
Let $f$ be the transnormal function (with respect to $g$) defined above.
According to Q-M Wang \cite{Wang}  for $t=f(P_x)$ and a regular value $c$  we have
$d(P_{x},f^{-1}(c))=\int_{c}^{t}\frac{ds}{\sqrt{b(s)}}.$ Since $f$ is also a transnormal function with respect to $\tilde{g}$, we conclude that $d(P_{x},C)=\tilde{d}(P_{x},C),$
for  $C=f^{-1}(c).$

Finally (f) can be proved using the fact that the integral curves of the gradient of a transnormal function are geodesic segments up to reparametrization;  see e.g. \cite{Wang}.
\end{proof}

\begin{theorem}[\cite{AlexToeben2,Molino}]
\label{lemma-metric-in-S}
There exists  a new metric $g_0$  on $\tub(P_q)$ so that the following holds:
\begin{enumerate}
\item[(a)] Consider the tangent space $T_{\tilde{q}}S_{\tilde{q}}$ with the inner product $g_{\tilde q}$ and $S_{\tilde{q}}$ with the metric $g_0$. Then  $\exp_{\tilde{q}}:T_{\tilde{q}}S_{\tilde{q}}\rightarrow S_{\tilde{q}}$ is an isometry.
\item[(b)] For this new metric $g_0$ we have that $\F|S_{\tilde{q}}$ and $\F|\tub(P_q)$ are also SRFs.
\item[(c)]  For each $x\in \tub(P_{q})$ the normal space of the leaf $L_x$ is tangential to the slice $S_{\tilde{q}}$ which contains $x$, where $\tilde q\in P_q.$
\end{enumerate}
\end{theorem}
\begin{remark}
\label{rem-prop-flat-metric-in-S}
A geodesic of the original metric starting orthogonally to $P_q$ remains an orthogonal geodesic with respect to the metrics $\tilde g$ and $g_0$
\end{remark}
\begin{proof}
Let $\tilde{g},\tilde{g}^{1},\tilde{g}^{2}$ be the metrics  defined in Theorem \ref{lemma-almost-product}. Let  $h_{\lambda}$ denote  the homothetic transformation with respect to $P_q$.
Define $g_{\lambda}=\frac{1}{\lambda^2}h^{*}_{\lambda} \tilde{g}^{1}+\tilde{g}^{2}.$ Note that  the metric $g_{\lambda}$ smoothly converges to  a metric $g_0$  for $\lambda\to 0$.
This metric $g_0$ restricted to $S_{\tilde{q}}$ is
the flat metric on $\nu P_{\tilde{q}}$ induced from $\tilde{g}^1$, where $\tilde{q}\in P_{q}.$
This implies that $\mathrm{L}_{\lambda}$ tends uniformly to $\mathrm{L}_{0}$, where $\mathrm{L}_{\lambda}$ is the length function for $g_\lambda$. It follows then that
\begin{equation}
\label{lemma-metric-in-S-Eq1}
\lim_{\lambda\rightarrow 0} d_{\lambda}(x,P)= d_{0}(x,P)
\end{equation}
where  $P$ is a plaque.
One can prove  that $\F$ is a SRF with respect to $g_{\lambda}$ using transnormal functions.

Finally let $x$ and $y$ be points which belong to the same plaque. Using equation \eqref{lemma-metric-in-S-Eq1} and the fact that $\F$ is a SRF with respect to $g_{\lambda}$ we conclude that

\begin{eqnarray*}
0&=&\lim_{\lambda\rightarrow 0} (\, d_{\lambda}(x,P)-d_{\lambda}(y,P)\,)\\
          &=& d_{0}(x,P)-d_{0}(y,P)
\end{eqnarray*}

The above equation implies  that the plaques are locally equidistant and hence that the singular foliation $\F$ is Riemannian.
\end{proof}

We will frequently use the above result so that is convenient to give a name to the induced foliation on the slice with the metric $g_0$.

\begin{definition}
\label{definition-infinitesimalfoliation}
Let $q \in M$ be a singular point and $S_q$ a slice at $q$. Following Theorem \ref{lemma-metric-in-S}
one can define a singular Riemannian foliation $\hat{\F}(q)$ on a $(T_qS_q,g_q)$ by
setting $\hat{\F}(q) = \exp_q^{\ast}\big(\F|_{S_q}\big)$ and extending to the whole space by homothetic transformations with respect to the origin. It is called
the \emph{infinitesimal foliation} or the \emph{linearization of $\F$} at $q$.
\end{definition}

The next proposition is a consequence of Theorem \ref{lemma-metric-in-S}, Remark \ref{rem-prop-flat-metric-in-S} and arguments in Euclidean space.

\begin{proposition}[\cite{AlexToeben2}]
\label{lemma-slice-fundamental}
Let $S_{q}$ be a slice at $q$ and
  $\varphi:S_{q}\rightarrow S_{q}$ be the geodesic symmetry at $q,$ i.e., $\varphi=\exp_{q}\circ(-\id)\circ\exp_{q}^{-1}$. Then
the map $\varphi$ respects the foliation $\F|S_{q}$, or in other words, the foliation $\F|S_{q}$ is invariant under the involution $\varphi$.
\end{proposition}
Note that $\varphi$ is equal to the homothety $h_{-1}$. In other words, the foliation is invariant under homotheties $h_\lambda$ for all $-1\leq\lambda\leq 1$.

From the above result we can infer the next  one.

\begin{corollary}[\cite{AlexToeben2}]
\label{singular-points-isolated}
Let $\gamma$ be a geodesic orthogonal to a regular leaf of a SRF. Then the singular points are isolated on $\gamma.$
\end{corollary}

We conclude this section by discussing some  results about a stratification of $M$.

The union of leaves with the same dimension is a disjoint union of embedded submanifolds called \emph{stratum}. The collection of all strata is a stratfication in the usual sense, see e.g. \cite{Duistermaat}. A {\em minimal stratum} is the stratum of leaves of least dimension. It is easy to see that it is a disjoint union of  isolated closed submanifolds, see \cite{Molino}. We note that each geodesic of the ambient manifold that starts orthogonally to a leaf in the stratum but tangentially to the stratum remains in this stratum.
Therefore the induced foliation on each stratum is a Riemannian foliation.
The observation that every geodesic perpendicular to a stratum is perpendicular to the leaves allows us to adapt the argument of Proposition \ref{homothetic-lemma} and conclude that the foliation is locally invariant by the homothetic transformations with respect to the stratum.



A natural question is the following. Assume that $\F$ is a singular foliation on $M$ and there there exists a metric $\tilde{g}$ so that the restriction of $\F$ to each stratum is a Riemannian foliation. Is $\F$ a SRF with respect to $\tilde{g}$?  This question is still open. A partial answer is the following proposition.

\begin{proposition}[\cite{Alex6,Molino}]
\label{prop-metric-stratum}
Let $\F$ be a SRF with respect to a metric $g$. Let $\tilde{g}$ be another metric with the following property: $(\mathcal M,\F|_{\mathcal M})$ is a Riemannian foliation with respect to $\tilde{g}$ for each stratum $\mathcal M$. Then $\F$ is a SRF with respect to $\tilde{g}.$
\end{proposition}

\begin{remark}[\cite{Alex6}]
\label{rem-prop-metric-stratum}

For the proof of the above proposition one uses only the following properties that are fulfilled by a SRF:
\begin{enumerate}
\item[(a)] Each stratum is a union of submanifolds.
\item[(b)] The sequence of plaques $\{P_{x_{n}}\}$ converges to $P_{x}$ when the sequence $\{x_{n}\}$ converges to $x$.
\item[(c)] The restriction of $\F$ to each stratum is a Riemannian foliation.
\item[(d)] For each vector $v$ that is not tangential to a stratum $\mathcal M$ there exists a curve $\beta:[0,\epsilon)\rightarrow M-\mathcal M$ and a distribution $t\rightarrow V_{t}$ along $\beta$  such that
 \begin{enumerate}
 \item[(d.1)]$T_{\beta(t)}P_{\beta(t)}\subset V_{t},$
  \item[(d.2)]$v$ is orthogonal to $V_{0}.$
\end{enumerate}
\end{enumerate}
That means, \emph{a singular foliation $\F$ on a Riemannian manifold $(M,g)$ with the above properties is in fact a SRF with respect to $g$.}
We will use this fact in  the proof of Theorem \ref{thm-Blowup-srf}

\end{remark}

\section{Bott-equifocality}\label{sec:Bottequifocality}

Let  $G\times M\to M$ be a proper isometric action  and $G(x)$ a \emph{principal orbit}, i.e.,
its isotropy representation is trivial. Thus one can extend each normal vector $\xi\in\nu_{x}G(x)$
to an equivariant normal vector field $\xi$ along $G(x)$. By the $G$-equivariance of the normal exponential map of $G(x)$ we  have that $\{\exp_{y}(\xi(y)) | y\in G(x)\}$ is an orbit of the action. Moreover, the action on $M$ is determined by the action on the normal bundle.

Similarly, for a leaf $L$ of a SRF without holonomy we can characterize a normal vector field of $L$ that maps to a leaf under the normal exponential map. As we will see in this section, it is a restriction of a basic vector field, which corresponds to the equivariant field in the case above, along $L$.

The \emph{Bott} or \emph{basic connection} $\nabla$ of the Riemannian foliation $\F|\mathcal M_r$ on the regular stratum $\mathcal M_r$ is the connection of $\nu L$ defined by $\nabla_XY=[X,Y]^{\nu}$ whenever $X$ is a vector field tangential to $\F$ and $Y$ is vector field of the normal bundle of the foliation. Here the superscript $\nu$ denotes projection onto $\nu\F$. Locally a Bott-parallel field is a basic field. The restriction of this connection to the normal bundle of a regular leaf is flat.


For a regular Riemannian foliation  we already know that on a complete manifold $M$ \emph{for a normal field $\xi$ parallel with respect to the Bott connection the curve $t\mapsto \exp(\xi\circ\beta)$ is  contained in the leaf $L_{\exp(\xi_{\beta(0)})}$ for any curve $\beta:[0,1]\rightarrow L$.  }

As we will see in Theorem \ref{thm-s.r.f.-equifocal} below, this property still holds even for singular Riemannian foliations and will be called
\emph{Bott-equifocality}. It allows us to give an alternative definition of the usual holonomy map.

\begin{definition}
\label{definition-holonomy-map}
Let $L$ be a regular leaf, $[\beta]$  the class of homotopies of a curve $\beta$ contained in $L$ fixing end points  and $\sigma_{\beta(i)}=\{\exp(\xi_{\beta(i)}) | \xi\in \nu_{\beta(i)} L, \|\xi\|<\epsilon\}$ for $i=0,1$ a transversal submanifold to the foliation at $\beta(i)$, where $\epsilon>0$ such that $\exp|\nu_{\beta(0)}^\epsilon L$ is a diffeomorphism.  Then the map
 $\varphi_{[\beta]}: \sigma_{\beta(0)}\rightarrow \sigma_{\beta(1)}$ defined as
$$\varphi_{[\beta]}(\exp(\xi(0)))=\exp(\xi(1)),$$
is called \emph{a (normal) holonomy map.} Here  $t\to \xi(t)$ is a normal field along $\beta$ parallel with respect to the Bott connection.
\end{definition}
\begin{remark}
Due to equifocality we have  $\varphi_{[\beta]}(x)\in L_{x}$ for all curves $\beta$ as above. Hence,
for a small $\epsilon,$ the map $\varphi_{[\beta]}$ coincides with the usual holonomy map. Also note that $\sigma_{\beta(i)}$ can contain singular points. This turns out to be very useful in the study of the Weyl group/pseudogroup of polar foliations as we will see in the next section.
\end{remark}

\begin{theorem}[\cite{AlexToeben2} Bott-equifocality]
\label{thm-s.r.f.-equifocal}
Let $\F$ be a SRF on a complete Riemannian manifold $M$. Then for each  regular point $p$ there exists a neighborhood $P_p$ of $p$ in $L_{p}$ such that the following holds:
\begin{enumerate}
\item[(1)] For each   normal foliated vector field $\xi$  along $P_p$  the derivative of  the end point map    $\eta_{\xi}:P_p\rightarrow M,$ defined as $\eta_{\xi}(x):=\exp(\xi_{x}),$ has constant rank.
\item[(2)] $W:=\eta_{\xi}(P_p)$ is an open subset of $L_{\eta_{\xi}(p)}$.
\end{enumerate}
\end{theorem}
\begin{remark}
The  same result and proof are valid for singular leaves, if one considers foliate vector fields tangential to the stratum.
\end{remark}

 The above theorem implies that one  can reconstruct the  (singular) foliation by taking all Bott-parallel submanifolds of a regular leaf with trivial holonomy.

\begin{corollary}[\cite{AlexToeben2}]
\label{cor-equifocality}
Let $L_{p}$ be a regular  leaf with trivial holonomy and $\Xi$ denote the set of all  normal foliate vector  fields along $L_{p}.$
\begin{enumerate}
\item[(1)] Let $\xi\in \Xi$. Then $\eta_{\xi}:L_{p}\rightarrow L_{q}$ is a covering map if $q=\eta_{\xi}(p)$ is a regular point.
\item[(2)]  $\F=\{\eta_{\xi}(L_{p})\}_{\xi\in \, \Xi},$ i.e., we can reconstruct the singular foliation by taking all Bott-parallel submanifolds of the regular leaf $L_{p}.$
\end{enumerate}
\end{corollary}

\begin{proof}

One should first verify  Theorem \ref{thm-s.r.f.-equifocal} for some tubular neighborhood $\tub(P_{q})$ for each point $q$. Then the  theorem is proved by covering a horizontal geodesic with such tubular neighborhoods; note that a horizontal geodesic through a regular point meets singular leaves at isolated times due to Corollary \ref{singular-points-isolated}. Then the result follows by applying the local result successively.
In order to obtain a local version of the result it suffices to have an infinitesimal one. More precisely it suffices to prove the next lemma.


\begin{lemma}
\label{lemma-jacobi-field-itens}
Let $q$ be a singular point and $x_0$ in a small tubular neighborhood $\tub(P_{q})$ and let $\xi$ be the appropriate parallel normal vector field of the plaque $P_{x_{0}}$ with $\exp(\xi_{x_0})=q$.
The geodesic $\gamma(t):=\exp((t+1)\,\xi_{x_{0}})$ lies in $\tub(P_{q})$ for all $t\in[-1,1]$ and $q=\gamma(0)$ is the only singular point on $\gamma.$
Define $f(s,t)=\exp(t\xi_{\alpha(s)})$ and $J(t)=\frac{\partial f}{\partial s}(0,t),$ where $\alpha(s)$ is a curve in $P_{x_0}.$
Then the Jacobi field $J$ is always tangential to the leaves.
\end{lemma}

To prove the lemma, let $g_{0}$ be the metric defined in Theorem \ref{lemma-metric-in-S}.
Then Remark \ref{rem-prop-flat-metric-in-S} implies that $J$ is still a Jacobi field for the new metric.


We define $\sigma$ as the submanifold contained in $S_{q}$ which is the image by $\exp_q$ of a ball in a subspace and so that $\sigma$ is orthogonal to $L_x$ at $x\in S_q.$

By Proposition \ref{lemma-slice-fundamental},  Theorem \ref{lemma-metric-in-S} and Proposition \ref{homothetic-lemma} the plaques $P_{\gamma(t)}\cap S_{q}$ are orthogonal to $\sigma$ for $-1\leq t\leq 1$. Then it follows from Theorem \ref{lemma-metric-in-S} that the plaques $P_{\gamma(t)}$ are orthogonal to $\sigma$ for $-1\leq t\leq 1.$

For a fixed $t$ consider a geodesic segment $\beta$ so that $\beta(0)=\gamma(t)$ and $\beta$ is orthogonal to $P_{\gamma(t)}$. Then Theorem \ref{lemma-metric-in-S} implies that $\beta$ is contained in $S_{q}.$ Since $S_q$ is identified with $T_{q}S_{q}$  we can consider $\beta$ as a straight line. Since $P_{\gamma(t)}\cap S_{q}$ is orthogonal to $\sigma$, and $\sigma$ is identified with a ball in a subspace, we conclude that $\beta$ is contained in $\sigma$.

Therefore $\exp_{\gamma(t)}(\nu P_{\gamma(t)}\cap B_{\epsilon}(0))$  is an open set of $\sigma$ and the second fundamental form is null at $\gamma(t)$, i.e., $\sigma$ is geodesic at $\gamma(t).$ In particular the curvature tensor $R$ of $\sigma$ is the same as the ambient space at $\gamma(t).$ This and the fact that $R(\gamma',\cdot)\gamma'$ is self-adjoint imply that
$T_{\gamma(t)}\sigma$ as well $(T_{\gamma(t)}\sigma)^{\perp}$ are families of parallel subspaces along $\gamma$ which are invariant under $R(\gamma',\cdot)\gamma'.$

Finally note $J(1),J'(1)\in (T_{\gamma(1)}\sigma)^{\perp}$. Thus the Jacobi equation implies $J(t)\in (T_{\gamma(t)}\sigma)^{\perp}$ for $-1\leq t\leq 1.$

As remarked above plaques $P_{\gamma(t)}$ are orthogonal to $\sigma$ for $-1\leq t\leq 1.$ Since $P_{\gamma(t)}$ are regular plaques for $t\neq0$  we conclude that $J(t)$ is always tangential to $P_{\gamma(t)}.$ This concludes the proof of Lemma \ref{lemma-jacobi-field-itens}.

\end{proof}

Bott-equifocality allows us to prove the next proposition that describes how the  plaques ``slide along the stratum". This proposition will be used in the last section.

\begin{proposition}[\cite{Alex6}]
\label{slidingalongStratum}
Let $\F$ be a SRF, $q$ a singular point and $\mathcal M_q$ the stratum containing $q$. For each $x\in P_{q}$ set $T_{x}:=\exp_{x}(\nu \mathcal M_p\cap B_{\epsilon}(0))$.
 Then  there exists a neighborhood $U$ of $q$ in $\mathcal M_p$ such that if a plaque $P$ of $\F$ meets $T_{q}$, we have
 $\emptyset\neq P\cap S_{x}\subset T_{x}$ for all $x\in U\cap P_{q}.$
  \end{proposition}

\begin{remark}
\label{remark-wilking-distribution}
Let $V$ denote a vector space of Jacobi fields orthogonal to a geodesic $\gamma$. Assume $g(J_1(t),J_{2}'(t))=g(J_{1}'(t),J_{2}(t))$ for all $J_1, J_2\in V$. Set
$$V(t):=\{J(t)| J\in V\}\oplus \{J'(t)| J\in V, J(t)=0\}\subset  T_{\gamma(t)}M $$
As remarked by Wilking \cite{WilkGAFA} (see also Gromoll and Walshap \cite[Lemma 1.7.1]{Gromoll-Walshap}) $\dim V(t)=\dim V$ and the second summand is trivial for almost every $t$.

If $\gamma$ is a geodesic orthogonal to a SRF $\F$ then, due to equifocality, one can choose $V$ such that $V(t)$ is the tangent space of $L_{\gamma(t)}$, if $\gamma(t)$ is a regular point. Therefore we can define $H(t)$ as the orthogonal space to $V(t)$  and this distribution is defined for all $t$ and is always orthogonal to the leaves of $\F$.
\end{remark}

\section{Polar foliations}\label{sec:polar}

\subsection{Definition and examples}

A singular Riemannian foliation $\F$ is called  a \emph{polar foliation} (or a \emph{singular Riemannian foliations with sections})
if, for each regular point $p$, there is an immersed submanifold $\Sigma_{p}$, called \emph{section},  that passes through $p$ and that meets all the leaves and always perpendicularly. It follows that $\Sigma_p$ is totally geodesic and that the dimension of $\Sigma_{p}$ is equal to the codimension of the regular leaf $L_{p}$. This is equivalent to saying that  the normal distribution of the regular leaves of  $\F$  is integrable; see Theorem \ref{teo-integra-polar}.

A typical example of a polar foliation is the partition  of a Riemannian manifold into the orbits  of a \emph{polar action}, i.e., an isometric action with sections.


Another important class of  examples  is the partition of a Euclidean space into the parallel submanifolds  of an isoparametric submanifold $N.$
Recall that a submanifold $N$ of a Euclidean space is called \emph{isoparametric}
if its normal bundle is flat and the principal curvatures along any
parallel normal vector field are constant. Note that there are  examples of inhomogeneous isoparametric submanifolds,
see Ferus, Karcher and Munzner~\cite{FerusKarcherMunzner}; a comprehensive introductionto the theory of isoparametric submanifolds can be found in the book of Berndt, Console and Olmos \cite{BerndtConsoleOlmos} and in Terng and Palais \cite{PalaisTerng}.

As we will see in the \emph{slice theorem}
(Theorem \ref{teo-slice-polar}), if
 $\F$ is polar then the infinitesimal foliation $\hat{\F}(q)$ is polar and hence isoparametric.

An isoparametric submanifold in a Euclidean space is a particular example of an equifocal submanifold; the latter notion was introduced by  Terng and Thorbergsson \cite{TTh1}.

A connected immersed submanifold $L$ of a complete Riemannian manifold
$M$ is called \emph{equifocal} if it satisfies the following
conditions:
\begin{enumerate}
\item The normal bundle $\nu L$ is flat.
\item $L$ has sections, i.e.~for each~$p\in L$,
the set $\Sigma :=\exp_{p}(\nu_p L_{p})$ is a complete immersed totally
geodesic submanifold.
\item For each parallel normal field $\xi$ on a neighborhood $U \subset L$,
the derivative of the end point map $\eta_{\xi}:U\to M$ defined
by $\eta_{\xi}(x):=\exp(\xi_{x})$ has constant rank.
\end{enumerate}

Note that the regular leaves of a polar foliation are equifocal. This follows from Theorem \ref{thm-s.r.f.-equifocal}  and from the fact that the normal connection and the Bott-connection coincide because of Theorem 5.5.12 in \cite{PalaisTerng}; see also \cite{Alex2}.

\begin{remark}
The converse, that the partition of a Riemannian manifold into parallel submanifold of an equifocal submanifold is a polar foliation, is also true under an additional assumption. More precisely,  T\"{o}ben~\cite{Toeben} proved the following result:
Let $L$ be a closed embedded equifocal submanifold with a globally flat normal bundle.  Suppose that set of regular points of each section $\Sigma$ is an open and dense set in $\Sigma.$ Define  $\Xi$ as the set of all parallel normal fields along $L.$ Then $\F:=\{\eta_{\xi}(L)\}_{\xi\in \, \Xi}$ is a polar foliation.
 Here  a point of a section is called regular if there exists only one local section  that contains $p,$ i.e., any two local sections $\sigma$ and $\tilde{\sigma}$  that contain $p,$ have the same germ at $p.$  See  \cite{Alex3} for an alternative proof of this result.
\end{remark}

\begin{remark}\label{rem:EquifocalWeyl}
As we have remarked in the last section, the equifocality of the leaves of $\F$ allows us to give an alternative definition of holonomy maps. We remark that their domains do not depend on the choice of a tubular neighborhood of the leaves and may contain singular points. This allows us to define a group action on the sections, which coincides with the Weyl group action for polar actions and isoparametric submanifolds; see the discussion in Subsection \ref{subsection-weylgroup}.
\end{remark}

Finally, let us remark that more examples of polar foliations can be constructed by
a suitable change of the metric and surgery (see \cite{AlexToeben}) or by a
suspension of a homomorphism (see \cite{Alex2,Alex4}). In particular, by means of a suspension of a homomorphism, one can also construct an example of a polar foliation $\F$ on a non-simply connected space with leaves that are not embedded and so that the partition formed by the closure of the leaves of $\F$ is a singular Riemannian foliation. This kind of example illustrates Molino's conjecture that the closures of the leaves of a SRF is again a SRF, see \cite{Alex4} for the proof of this conjecture in the case of polar foliations.



\subsection{Integrability and the slice theorem}
\label{subsection-intebrability-slice}

In this section we prove the \emph{slice theorem} (Theorem \ref{teo-slice-polar}).
To prove this result we  need two key ingredients: an alternative definition of polarity via integrability (Theorem \ref{teo-integra-polar}), and a weaker version of the slice theorem for integrable foliations due to Boualem (Theorem \ref{teoBoualem}).

An \emph{integrable singular Riemannian foliation} (ISRF)  is a SRF such that horizontal distribution is integrable on the regular stratum. Therefore one has a regular foliation $\F^{\perp}$, whose leaves are then automatically totally geodesic, on the regular stratum $\mathcal M_r$ determined by the horizontal distribution. A \emph{regular local section} for $\F$ is an open subset of some $\F^{\perp}$-leaf.


In what follows we review a result due to Boualem \cite{Boualem}, following arguments of Caceres \cite{diego}.

\begin{theorem}[\cite{Boualem}]
\label{teoBoualem}
Let $\F$ be an ISRF  on a complete manifold $(M,\metric)$. For any singular point $q \in M$, let $\sigma \subseteq B_{\varepsilon}(q)$ be a regular local section and $S_q = \exp_q^{\nu}(\mathcal{B}_{\varepsilon}(0))$ the $\varepsilon$-slice at $q$.
Assume that the intersection $\sigma \cap S_q$ is nonempty. Then $\sigma$ is an open subset of $\exp_q^{\nu}(\mathcal{B}_{\varepsilon}(0)\cap V)$ for some vector space $V \subset T_q S_{q}$. Furthermore, the linearization $\hat{\F}(q)$ is an ISRF with $\hat{\sigma} := \exp_q^{-1}(\sigma)$ as a regular local section.
\end{theorem}

\begin{proof}
We begin with a lemma.



\begin{lemma} \label{lema1-boualem}
If  $\sigma $ intersects $S_q$ at a point $p$ then $T_p\sigma \subseteq T_pS_q.$
In addition there exists a subspace $V\subset T_{q}S_{q}$  such that $d(\exp_q)_xV= T_{p}\sigma$ for $p=\exp_q(x)$.
\end{lemma}
To prove the above lemma, let $V$ be the affine space such that $d (\exp_{q})_{x} V= T_{p}\sigma.$ Note that $V$ is a linear subspace because if $\gamma$ is a geodesic segment with $\gamma(0)=q$ and $\gamma(1)=p$ it must be orthogonal to $L_{p}$ and hence tangential to $\sigma$. Set  $N:=\exp_q(\mathcal{B}_{\varepsilon}(0)\cap V)$ and consider a Jacobi field $J$ along $\gamma$ and tangent to $N$. Since $\sigma$ is totally geodesic, one can prove that $J$ splits into two Jacobi fields, one $J_{\sigma}$ tangential to $\sigma$ and the other $J_{\perp}$ orthogonal. Those Jacobi fields are defined along $\gamma$ even at $t=0$. Since $J_{\perp}(1)=0$ and $\gamma$  is contained in a normal neighborhood we conclude that $J_{\perp}=0$. Therefore $J$ is tangential to $\sigma$. From this we conclude that $L_{\gamma(t)}$ is orthogonal to $N$ for $t>0$ and hence, by continuity $L_{q}$ is orthogonal to $N$. In other words $N\subset S_{q}$ and this finishes the proof of Lemma \ref{lema1-boualem}.

Now note that $\sigma$ is contained in $S_{q}$, because, by Lemma \ref{lema1-boualem}, $\sigma$ is tangential to the leaves of the foliation of $B_{\varepsilon}(q)$ by slices $\{ S_x \mid x \in P_q \}.$
Let $\metric_{\lambda}$ be the family of metrics defined in the proof of Theorem \ref{lemma-metric-in-S} and recall that $\metric_{\lambda}$ smoothly converges to $\metric_{0}$. Also note that regular sections and plaques are invariant under the homothety $h_{\lambda}$ thus $\F|_{S_q}$ on $(S_q,\metric_{\lambda})$ is an ISRF  (with the same regular sections) for every $\lambda > 0$. Therefore the linearization  $\hat{\F}(q)$ is an ISRF.
By Lemma \ref{lema1-boualem}, the totally geodesic submanifold $\hat{\sigma}$ is tangential to the subspace $V$. Therefore $\hat{\sigma}$ is an open subset of $V$
 and this concludes the proof of Theorem \ref{teoBoualem}.
\end{proof}

It is natural to ask if an ISRF is a polar foliation. In fact this question was proposed by
Palais and Terng~\cite[Remark 5.6.8]{PalaisTerng} for the case that $\F$ is a partition of $M$ into orbits of isometric action.
The next result gives a positive answer to this question for the more general class of singular Riemannian foliations.

\begin{theorem}[\cite{Alex4}]
\label{teo-integra-polar}
Let $\F$ be an ISRF on a complete Riemannian manifold $M.$  Then $\F$ is a polar foliation   and regular points are open and dense in each section.
\end{theorem}
\begin{remark}
Particular cases of the above result where proved
by Molino and Pierrot~\cite{MolinoPierrot},
Boualem~\cite{Boualem} and Lytchak and Thorbergsson~\cite{LytchakThorbergsson1}. Heintze, Liu and Olmos~\cite{HOL} proved the above result when $\F$ was a partition of $M$ into orbits of an isometric action.
Szenthe~\cite{Szenthe} also worked on the question for isometric actions.
\end{remark}
\begin{proof}

We give a sketch of the proof and begin with a lemma. We retain the notation of Theorem \ref{teoBoualem}.
\begin{lemma}\label{lemmaslice}
We have:
\begin{enumerate}
\item[(a)]  The set of regular points of $\widehat{\F}(q)$ in $\mathcal{B}_{\varepsilon}(0)\cap V$ is open and dense; it is orthogonal to the leaves of $\widehat{\F}(q)$.
\item[(b)]The set of regular points of $\F$ in $D_{\epsilon}(q):=\exp_q(\mathcal{B}_{\varepsilon}(0)\cap V)$ is open and dense.
\item[(c)]The submanifold $D_{\epsilon}(q)$ is orthogonal to the leaves of $\F.$
\item[(d)] $D_{\epsilon}(q)$ is a totally geodesic submanifold.
 \end{enumerate}
\end{lemma}
  The fact that the set of regular points of $\widehat{\F}(q)$ in $\mathcal{B}_{\varepsilon}(0)\cap V$ is open and dense follows directly from Proposition \ref{singular-points-isolated}. To prove that $\mathcal{B}_{\varepsilon}(0)\cap V $ is orthogonal to the plaque $\widehat{P}_{x}$ at $x\in \mathcal{B}_{\varepsilon}(0)\cap V,$ define
straight lines $l_{1},\ldots,l_{k}:[0,1]\rightarrow V$ such that
\begin{enumerate}
\item $l_{i}(0)\in\hat{\sigma},$ where $\hat{\sigma}$ is a regular section contained in $V$,
\item $l_{i}(1)=x,$
\item $\{l_{i}^{'}(1)\}$ is a basis of $T_{x}V.$
\end{enumerate}
Since $\widehat{\sigma}$ is a regular section and $\widehat{\F}$ is a singular Riemannian foliation,  the straight lines $l_{i}$ are orthogonal to the plaque $\widehat{P}_{x}$. This concludes the proof of (a). Item (b) follows directly from (a). Item (a) and  Theorem \ref{teoBoualem} imply (c). Item (d) follows from (c) and from the fact that regular points are dense in $D_{\epsilon}(q)$ and this finishes the proof of the lemma.

Now let $B_{\delta}(q)$ be a normal ball of radius $\delta$ in $M$ centered at $q.$ Note that $\delta$  can be greater than $\varepsilon$, the radius  of the  slice $S_{q}$, which depends on the leaf $L_{q}.$ Using the above lemma and the fact that $\exp_{q}$ is a diffeomorphism on $\mathcal{B}_{\delta}(0)$
we can replace the radius $\epsilon$ of $D_{\epsilon}(q)$  with the radius $\delta$ that does no longer depend on the leaves of $\F$. More precisely, one can prove the next lemma.
\begin{lemma}
\label{cor-discos-geodesicos}
 Consider the disc $D_{\delta}(q)=\{\exp_{q}(v)\mid v\in V\ \mbox{and}\ \|v\|<\delta\}$. Then
\begin{enumerate}
\item[(a)]the set of regular points of $\F$ in $D_{\delta}(q)$ is open and dense.
\item[(b)]The submanifold $D_{\delta}(q)$ is orthogonal to the leaves of $\F.$
\item[(c)] $D_{\delta}(q)$ is a totally geodesic submanifold.
\end{enumerate}
\end{lemma}

Let $r_{1}$ be a fixed positive number. Then we can find a $\delta>0$ such that $B_{\delta}(x)$ is a normal ball for each point $x$ in a neighborhood of the compact set  $D_{r_{1}}(q)$.  Let $\gamma$ be a geodesic of length $r_{1}$ such that  $\gamma'(0)\in T_{q} D_{\epsilon}(q).$

Note that we can extend $D_{\epsilon}(q)$ along the geodesic $\gamma$ by covering $\gamma$ with totally geodesic discs from the lemma. We obtain a noncomplete submanifold which satisfies the properties of Lemma \ref{cor-discos-geodesicos}.

Let $\Sigma_{r_{1}}$ denote the union of the extensions constructed above along every geodesics $\gamma$ with length $r_{1}$ such that $\gamma '(0)\in T_{q}D_{\epsilon}(q).$

Then one can show
\begin{lemma}
$\Sigma_{r_{1}}$ can be covered by a finite number of discs $D_{i}$ (of radius $\delta$) which satisfy (a), (b) and (c) of Lemma \ref{cor-discos-geodesicos}.
\end{lemma}

The above lemma assures that $\Sigma_{r_{1}}$ is an immersed submanifold. Indeed, let $\pi:G_{k}(TM)\to M$ be the Grassmann bundle of $k$-planes of $T M$ over $M$,
$D^{\star}_{i}:=\{T_{x}D_{i}\in G_{k}(TM)\mid x\in D_{i}\}$  and
$\Sigma^{\star}_{r_{1}}:=\{T_{x}D_{i}\in G_{k}(TM)\mid  x\in \Sigma_{r_{1}}\}.$
Clearly $D^{\star}_{i}$ is embedded. The fact that $D_{i}$ is totally geodesic implies that, if  $D^{\star}_{i}$ and  $D^{\star}_{j}$ have a common point then they coincide in a neighborhood of this point. Since $\Sigma^{\star}_{r_{1}}$  can be covered by a finite number of $D^{\star}_{i},$  we see that     $\Sigma^{\star}_{r_{1}}$  is  embedded. Therefore  $\Sigma_{r_{1}}=\pi(\Sigma_{r_{1}}^{*})$ is an immersed submanifold.

Finally set $\Sigma^{\star}:=\cup_{r_{i}}\Sigma^{\star}_{r_{i}}$ where $\{r_{n}\}$ is a sequence with $r_{n}\to \infty.$ The above discussion and Hopf and Rinow's theorem  imply that $\Sigma:=\pi(\Sigma^{\star})$ is a complete immersed totally geodesic submanifold.

\end{proof}


We are ready to state the \emph{slice theorem}.

\begin{theorem}[\cite{Alex2}]
\label{teo-slice-polar}
Let $\F$ be a polar foliation on a complete Riemannian manifold $M.$
Let $q$ be a singular point of $M$ and let
$S_{q}$ be a slice at $q$ of radius $\epsilon$. Then
\begin{enumerate}
\item[(a)] $S_{q}= \cup_{\sigma\in\Lambda (q)}\, \sigma$,
where $\Lambda(q)$ is the set of
local sections $\sigma$ containing $q$ such that $\dist(p,q)<\epsilon$
for each $p\in\sigma.$
 \item[(b)] $S_{x}\subset S_{q}$ for all $x\in S_{q}$.
\item[(c)] $\F|S_q$ is a polar foliation \ on $S_{q}$ with respect to the induced metric from
$M$.
\item[(d)] The infinitesimal foliation $\hat{\F}(q)$ is polar and hence isoparametric.
\end{enumerate}
\end{theorem}
\begin{proof}

The main ingredients in the proof of the first three items are the fact that sections
are totally geodesics and the set of regular points in a section is open and dense in this section.
Let us now explain how item (d) can be proved.
From Lemma \ref{lemmaslice} follows that $\hat{\F}(q)$ is a polar foliation of Euclidean space. Therefore each regular leaf is equifocal, and hence by a straightforward calculation, isoparametric. We fix a regular leaf $L$. The isoparametric foliation given by the (Bott-) parallel submanifolds of $L$ and thus coincides with $\hat\F_q$ by Theorem \ref{thm-s.r.f.-equifocal}.


\end{proof}


\subsection{Weyl group}
\label{subsection-weylgroup}

As explained in Remark \ref{rem:EquifocalWeyl} due to Bott-equifocality it is possible to define a map $\varphi_{[\beta]}: \sigma_{\beta(0)}\rightarrow \sigma_{\beta(1)}$
such that $\varphi_{[\beta]}(x)\in L_{x}$, where $\sigma_{\beta(i)}$ is an open set of a section possibly containing  singular points (recall Definition \ref{definition-holonomy-map}).
For a given section $\Sigma$ the \emph{Weyl pseudogroup} of $\Sigma$ is defined as
as the pseudogroup generated by the local isometries $\varphi_{[\beta]}$ such that $\beta(0),\beta(1)\in \Sigma$. By the appropriate choice of a section $\Sigma$, this pseudogroup turns out to be a group of (global) isometries of $\Sigma$ that is called the \emph{Weyl group} $W(\Sigma)$; see T\"{o}ben \cite{Toeben}.
From now on we fix such a section $\Sigma$.

By definition, if $w\in W(\Sigma)$ then $w(x)\in L_{x}$ and $W(\Sigma)$ describes how the leaves of $\F$ intersect the section $\Sigma$. As expected, when $\F$ is  the partition of a Riemannian manifold $M$ into the orbits of a polar action $G\times M\to M$, the Weyl group $W(\Sigma)$ is the usual Weyl group $N/Z$, where $N=\{g\in G | \, g\cdot x\in \Sigma, \forall x\in \Sigma \}$ and $Z=\{g\in G | \, g\cdot x=x, \forall x\in \Sigma \}$; see Palais and Terng \cite{PalaisTerng}.

Using the slice theorem (Theorem \ref{teo-slice-polar}) it is possible to prove that
the intersection of the singular leaves of $\F$ with the section $\Sigma$ is a union of totally geodesic hypersurfaces, called \emph{walls}, and  the reflections in the walls are elements of the Weyl group $W(\Sigma)$; see \cite{Alex2}.
We  define  $\Gamma(\Sigma)$  as the group generated by these reflections.

It is natural to ask when $\Gamma(\Sigma)=W(\Sigma).$ As suggested by classical examples, this question is related to the question when the holonomy of the regular leaves are trivial.  These problems were studied in \cite{AlexToeben2,Alex5,Lytchak,Alex7}. More precisely in \cite{AlexToeben2,Alex5}  it was proved that if the leaves of $\F$ are closed embedded and if  $M$ is simply connected, then the regular leaves have trivial holonomy and $M/\F=\Sigma/W(\Sigma)$ is a simply connected Coxeter orbifold. In \cite{Lytchak} Lytchak proved that if $M$ is simply connected, then the leaves are closed embedded. Concise alternative proofs of these results can be found in \cite{Alex7}, in which a surjective homomorphism $\pi_{1}(M)\to W(\Sigma)/\Gamma(\Sigma)$ was constructed. In addition we have that $M/\F=\Sigma/W(\Sigma)$ is a simply connected good Coxeter orbifold $\widetilde\Sigma/\Gamma$ where $\widetilde\Sigma$ is the Riemannian universal cover of the section $\Sigma$  and $\Gamma$ is a reflection group on $\widetilde\Sigma$. Moreover  $W(\Sigma)=\Gamma/\pi_{1}(\Sigma);$ see \cite{Alex7}.


\section{Variationally complete foliations}\label{sec:VC}


Let $G$ be a Lie group acting by isometries on a complete Riemannian manifold $M$. We say a Jacobi field is \emph{$G$-transversal} if it is obtained as an infinitesimal variation of a family of geodesics all normal to the orbits. If every $G$-transversal Jacobi field that is tangential to orbits at two distinct points is the (restriction of) a Killing field induced by the $G$-action, we say the action is \emph{variationally complete}.

The concept of variational completeness can be generalised to the setting of singular foliations. Let $L$ be an immersed submanifold of $M$. An \emph{
$L$-Jacobi field} is a Jacobi field that is the infinitesimal variation of a family of geodesics all normal to $L$.

\begin{definition}
A SRF $\F$ is said to be \emph{without horizontally conjugate points} or simply \emph{variationally complete} (VC) if for every leaf $L$ every $L$-Jacobi field tangential to a leaf other than $L$ is vertical.
\end{definition}
It is straighforward to check the following.

\begin{proposition}[\cite{LytchakThorbergsson1}]
The foliation given by a partition  of $M$ into the orbtis of a  $G$-action is VC if, and only if, the action is variationally complete.
\end{proposition}

Bott originally introduced variational completeness in \cite{Bott56}, and with Samelson in \cite{Bott58}. In the latter paper authors wrote ``intuitively, we like to think of variational completeness as absence of conjugate points on the decomposition space'' (p. 965). This can be better understood by saying that VC is a transversal condition, i.e., it can be defined in terms of the quotient.

\begin{example} \label{exemplo-vc}
Let $\F$ be the foliation on $M$ given by the fibers of a Riemannian submersion $\pi : M \to B$. Then $\F$ is VC if and only if $B$ has no conjugate points.
\end{example}



The relation between variational completeness and polarity was explored (from the viewpoint of $G$-actions) by Conlon \cite{Conlon71,Conlon72} who showed hyperpolar actions of compact groups are variationally complete actions. The converse for actions on symmetric spaces was established by Gorodski and Thorbergsson \cite{ClaudioThor}. For foliations we have the following converse due to Lytchak and Thorbergsson \cite{LytchakThorbergsson1}.

\begin{theorem}[\cite{LytchakThorbergsson1}] \label{teoBvc}
A VC foliation on a complete manifold $M$ with non-negative sectional curvature is hyperpolar.
\end{theorem}

Before proceeding with a sketch of the proof of the above theorem, we recall some facts about an invariant of a Riemannian submersion $\pi : M \to B$. In \cite{oneill}, O'Neill introduced a horizontal tensor \begin{equation} \label{Aneill} A_X Y := \big( \nabla_{X^h} Y^h \big)^{v} + \big( \nabla_{X^h} Y^v \big)^{h},\end{equation} where superscripts $h$ and $v$ denote the projection on the horizontal and vertical spaces of the submersion, respectively.
It can be shown that
\begin{equation} \label{Aneint} A_X Y = \frac{1}{2}[X,Y]^v \end{equation} for horizontal vector fields $X, Y$. We see that $A \equiv 0$ for horizontal fields if and only if the submersion is integrable. Moreover, the sectional curvatures of $K_B$ and $K_M$ of $B$ and $M$, respectively, are related by
\begin{eqnarray}
K_B\big( d\pi_p (\sigma_p)\big) & = & K_M(\sigma_p) + 3 ||A_X Y||^2 \label{3eq1} \\
																& = & \label{neillkg}\metric\big((R(X,Y)X)^h - 3 A_X A_X Y ,Y\big),				
\end{eqnarray}
where the $2$-plane $\sigma_p \subset T_pM$ is spanned by $X$ and $Y$.
\begin{proof}
Let $\gamma:\RR\to M$ be a horizontal geodesic through a regular point. As remarked before, the singular points on $\gamma$ are isolated. Following Wilking \cite{WilkGAFA} (see Remark \ref{remark-wilking-distribution}) it is possible to extend the horizontal distribution $H$ and the vertical distribution $H^{\perp}$ over the singular points to distributions along $\gamma$. Moreover, the O'Neill tensor from \eqref{Aneill} can be extended from the regular set to a horizontal tensor along $\gamma$
\begin{equation} \label{eq-At} A_t (X(t)) := \big( (X^h)'(t) \big)^v + \big( (X^v)'(t) \big)^h .\end{equation}
Assume now that the horizontal distribution on the regular stratum is not integrable. We choose $\gamma$ such that $\gamma(0)$ is regular and $A_0v_0\neq 0$ for some $v_0$.
A vector field $Y$ along $\gamma$ is the horizontal component of a Jacobi field along $\gamma$ if \begin{equation} \label{eqJacWil} \nabla_t^h \nabla_t^h Y + (R(\gamma',Y)\gamma')^h - 3 A_t A_t Y = 0,\end{equation} where $\nabla_t^h$ is the induced connection on the horizontal bundle. Note the last two terms on the left-hand side are analogous to O'Neill's equation \eqref{neillkg}.

Choose $X_1,\cdots,X_l$ orthonormal, $\nabla_t^h$-parallel fields that together with $\gamma'(t)$ form a basis of the horizontal space at $\gamma(t)$. For a curve $v(t)=(v_1(t),\ldots,v_l(t))$ in $\RR^l$ we obtain a horizontal field $\Phi(v) := \sum v_iX_i$ along $\gamma$. Clearly, $\Phi$ is a bijection of curves in $\RR^l$ and the horizontal fields perpendicular to $\gamma$. For a curve $v$ in $\RR^l$ we define
\begin{equation} \label{def-Rcal} \mathcal{R}(t)v(t):= \Phi^{-1}\big(R(\gamma'(t), Y(t)) \gamma'(t))^h - 3 A_t A_t Y(t) \big),
\end{equation}
where $Y := \Phi\circ v$.

Suppose $M$ has non-negative curvature. Then \begin{equation} \label{eqmal2} \ip{\mathcal{R}(t)v(t)}{v(t)} \geq 0\end{equation}
for all $v$.
Like in the Morse-Index Theorem for geodesics, the self-adjoint operator $\mathcal{R}$ can be used to define an index form $I_a^b$ over a set of curves in $\RR^l$ as follows:
$$I_{a}^{b}(\alpha,\beta):=\int_{a}^{b} \langle \alpha'(t), \beta'(t)\rangle
-\langle \mathcal{R}(t)\alpha(t),\beta(t)\rangle d t$$

One can prove that $I_a^b$ has positive index for some \emph{large} interval $[a,b]$. In fact, since $A_0v_0\neq0$ and therefore $\ip{\mathcal{R}(0)v_0}{v_0}>0$  we find a curve $v$ with $v(0)=v_0$ and sufficiently small energy $\int_a^b\|v'\|^2$ on a large interval $[a,b]$ such that $I_a^b(v,v)<0$. Thus one can conclude the existence of a nontrivial solution $\alpha$ to $\alpha'' + \mathcal{R}\alpha = 0$ and it can be seen that $Y=\Phi\circ\alpha$ is the horizontal component of an $L_{\gamma(a)}$-Jacobi field with $J(a)=J(c)=0$ for some $c \in (a,b)$. This implies the existence of a horizontal geodesic with horizontally conjugate points, contradiction.

In other words, a VC foliation on a complete manifold $M$ with non-negative curvature is a ISRF, and therefore polar by Theorem \ref{teo-integra-polar}. To show hyperpolarity assume that there is a nonflat section. Then inequation \eqref{eqmal2} is valid and we can proceed as before.
\end{proof}


\section{Infinitesimally polar foliations}\label{sec:infpolar}


In this section we briefly discuss the work of Thorbergsson and Lytchak \cite{LytchakThorbergsson2} about infinitesimally polar foliations.


\begin{definition}
A singular Riemannian foliation is called {\emph infinitesimally polar} if, for each $p\in M,$ the infinitesimal foliation $\hat{\F}(p)$ is polar.
\end{definition}

Let $\mathcal M_r$ be the regular stratum of a SRF $\F$ on a complete Riemannian manifold $M.$
For $z \in \mathcal M_r$ let $\overline{k}(z)$ be the supremum of sectional curvatures at the projection of $z$ on the leaf space.

\begin{theorem}[\cite{LytchakThorbergsson2}] \label{thm-infpol}
The following are equivalent:
\begin{itemize}
\item[(i)] the infinitesimal foliation $\hat{\F}(p)$ of $\F$ at $p$  is polar;
\item[(ii)] $\lim\sup_{z \to p} \overline{k}(z) < \infty \, , \, z \in \mathcal M_r$;
\item[(iii)] there is a neighborhood $Q$ of $p$ in $M$ such that $\F|_{Q}$ is closed and $Q/\F$ is a Riemannian orbifold.
\end{itemize}
\end{theorem}

So foliations verifying any of the equivalent conditions for all point $p$ are infinitesimally polar. From Theorem \ref{teo-slice-polar} its follows directly that polar foliations are infinitesimally polar.

Singular Riemannian foliations of $\codim\F = 1$ are clearly (polar and) infinitesimally polar. If $\codim\F = 2$ then the foliation $\hat{\F}(p)$  is polar/isoparametric in $\nu_p L_p$ since its restriction $\hat{\F}(p)|\nu_p^1 L_p$ to the sphere is of codimension $1$ and therefore polar/isoparametric.

\begin{proof}
First note that a Riemannian orbifold has locally bounded curvature, so $(iii) \Rightarrow (ii)$.

Now $(ii)$ implies that \begin{equation}\label{1imp2eq1} \lim\sup_{z \to p} \bar{k}(z) \cdot d^{2}(z,p)=0 \, , \, z \in \mathcal M_r. \end{equation}

Let $\tilde{g}$ be the metric defined in Theorem \ref{lemma-almost-product}. It is well-known that the Levi-Civita connection does not change for the conformally changed metrics ${\tilde{\metric}}^{\lambda} := \lambda^{-2}{\tilde{\metric}}$. Hence one has for the respective sectional curvatures $\tilde{k}^{\lambda} = \lambda^2 \tilde{k}$. Let $\metric_{\lambda}$ be the metric defined in the proof of Theorem \ref{lemma-metric-in-S}. So the homothety $h_{\lambda}$ is an isometry from the slice $(S_p,\metric_\lambda)$ to $(h_{\lambda}(S_p),{\tilde{\metric}}^{\lambda})$ with $h_{\lambda}(S_p) \subset S_p$ for $0<\lambda<1$. For regular $z\in S_p$,
\begin{equation}\label{1imp2eq2} \overline{k_{\lambda}}(z) = \overline{\tilde{k}^{\lambda}}(h_{\lambda}(z)) = {\lambda^2}\overline{\tilde{k}}(h_{\lambda}(z)) = {\lambda^2}\overline{k}(h_{\lambda}(z)), \end{equation}
the last equality holds due to Theorem \ref{lemma-almost-product}(e).
On the other hand one sees that the smooth convergence $\mathrm{g}_{\lambda}\to \mathrm{g}_{0}$ (see Theorem \ref{lemma-metric-in-S}) is also true for the induced transverse metrics (on a neighborhood of a regular point) thus
\begin{equation}\label{1imp2eq3} \lim_{\lambda \to 0}  \overline{k_{\lambda}}(z) = \overline{k_0}(z). \end{equation}
The right side of equation \eqref{1imp2eq2} goes to zero for $\lambda\to 0$ because of equation \eqref{1imp2eq1}. Thus, by equation \eqref{1imp2eq3}, the leaf space of the infinitesimal foliation is flat on the regular set. Therefore, both curvature terms in equation \eqref{3eq1} are zero and hence the regular horizontal distribution is integrable. By Theorem \ref{teo-integra-polar}, $\hat{\F}(p)$ is polar proving $(ii) \Rightarrow (i)$.

Finally, assume $\hat{\F}(p)$ is polar and let $\Sigma$ be a horizontal section through the origin and $W$ its associated Weyl group. Let $N:=\exp_p(\Sigma\cap B_\varepsilon(0))\subset S_p$, where $\varepsilon$ is the radius of $S_p$.
We want to construct a metric on $N$ that coincides with the transverse metric of $g$ on the regular set.
For $z \in N$ let $$V_z := (T_zN)^{\perp_{g_0}}\,; \qquad H_z := (V_z)^{\perp_{\tilde\metric}}.$$
Note that the distribution $V$ is the tangent bundle of the foliation on the regular stratum and $H$ extends the normal bundle of the foliation on the regular stratum over the singular set in $N$. Also observe that the orthogonal projection $\mathrm{pr}_z: T_{z}M \to H_z$ with respect to $\tilde\metric$ (i.e. the projection along the leaves) induces an isomorphism  $\mathrm{pr}_{z}: T_{z}N \to H_z$
 and define ${\metric}_z^N(\cdot,\cdot) := \tilde\metric_z(\mathrm{pr}_z\cdot,\mathrm{pr}_z\cdot)$ on $N$.

Now consider $N$ endowed with the metric $\metric^N$.
Since $\metric^N$ on the regular set of $N$ coincides with the tranverse metric, it is left invariant under the action of $W$. As the regular points are dense in $N$ we conclude from the theorem of Myers-Steenrod, that $W$ is a subgroup of isometries of $(N,\metric^N).$
It follows that $N/W$ and $S_p /\F$ are isometric, making the local quotient of a tubular neighborhood a Riemannian orbifold. This amounts to $(i) \Rightarrow (iii)$ and concludes the proof.

\end{proof}



We now consider another example of infinitesimally polar foliations.

\begin{theorem}[\cite{LytchakThorbergsson2}] \label{teovc}
Variationally complete foliations are infinitesimally polar.
\end{theorem}


\begin{proof}
In order to prove Theorem \ref{teovc} it suffices to consider regular geodesics.
We call a (not necessarily complete) geodesic {\em regular} if it only contains regular points.
\begin{lemma}
If $\hat{\F}(p)$ is not polar, then there is a horizontal regular geodesic with horizontally conjugate points.
\end{lemma}

The proof of the above Lemma follows the proof of Theorem \ref{teoBvc}. Note that the curvature of the ambient space is nonnegative (even flat) as in Theorem \ref{teoBvc}. The key difference is that from Theorem \ref{thm-infpol} one knows that the curvature becomes arbitrarily large for certain paths thus ensuring that $I^b_a$ has positive index for a \emph{small} interval $(a,b)$.

\begin{lemma} \label{lema2-casovc}
Let $\F$ be VC on $(M,\metric)$. Then no horizontal regular geodesic of $\hat{\F}(p)$ has horizontally conjugate points.
\end{lemma}
Theorem \ref{teovc} follows from the two previous lemmas. We now sketch the proof of  Lemma \ref{lema2-casovc}. Consider the change of metrics $$\metric \longrightarrow \tilde{\metric} \longrightarrow \tilde{\metric}^{\lambda} \longrightarrow \metric_{\lambda} \longrightarrow \metric_0$$ on the slice $S_p$. We claim that under each change of metric the foliation restricted to the regular stratum is still variationally complete. This will conclude the proof of Lemma \ref{lema2-casovc}.

Indeed, by Theorem \ref{lemma-almost-product} the first change does not affect the transverse metric. The second change, from $\tilde{\metric}$ to $\tilde{\metric}^{\lambda}$, retains the Riemannian connection and orthogonality. The third change, from $\tilde{\metric}^{\lambda}$ to $\metric_{\lambda}$, does not affect variational completeness since $(S_p,\metric_{\lambda}) \stackrel{h_{\lambda}}{\longrightarrow} (h_{\lambda}(S_p),{\tilde{\metric}}^{\lambda})$ is an isometry. We consider the fourth change, from $\metric_{\lambda}$ to $\metric_0$. Consider the projection $\gamma_0$ of a horizontal regular geodesic with respect to $\metric_0$ to the regular leaf space. The geodesic $\gamma_0$ is the limit of a sequence $(\gamma_{\lambda})$ of geodesics in the regular leaf space with respect to $\metric_{\lambda}$ with the same initial conditions as $\gamma_0$. If $\gamma_0$ had conjugate points, so would $\gamma_\lambda$ for small $\lambda>0$, contradiction. This finishes the proof sketch of Lemma \ref{lema2-casovc}.

\end{proof}


We conclude this section discussing the structure of $M/\F$ when $\F$ is an infinitesimally polar foliation  on
 a complete Riemannian manifold $M$.

 By definition, for each point in $M$ we can find a neighborhood $U$, called an \emph{isoparametric neighborhood}, such that $\F|_{U}$ is diffeomorphic to an isoparametric foliation. Therefore we can find a submanifold $\sigma$ transverse to the plaques that we  call a \emph{local section}.

Using the same argument as in the proof of Theorem \ref{thm-infpol}, we  induce a metric on $\sigma_{\alpha}$ such that,
for each two points $p,q\in \sigma_{\alpha}$, the distance between these points (with respect to the metric of $\sigma_{\alpha}$) is the same as between the plaques $P_{p}$ and $P_{q}$ (with respect to the metric of $M$).

Now  consider an open covering $\{U_{\alpha}\}$ of $M$ by isoparametric neighborhoods.
If $U_{\alpha}\cap U_{\beta}\neq \emptyset$, then we  find neighborhoods $V^{\alpha}_{\alpha}\subset \sigma_{\alpha}$,
$V^{\alpha}_{\beta}\subset \sigma_{\beta}$ and an isometry $\varphi_{\beta,\alpha}:V^{\alpha}_{\alpha}\cap C_{\alpha}\to V^{\alpha}_{\beta} \cap C_{\beta}$
where $C_{\alpha}$ (respectively $C_{\beta}$) is a Weyl chamber of $\sigma_{\alpha}$ (respectively $\sigma_{\beta}$).
Since $C_{\alpha},$ $C_{\beta}$ are Weyl chambers, we  extend the map $\varphi_{\alpha,\beta}$ to an isometry
$\varphi_{\beta,\alpha} :V^{\alpha}_{\alpha}\to V^{\alpha}_{\beta}$, that is  called \emph{holonomy}.

The elements $\varphi_{\beta,\alpha}$ acting on $\sigma=\amalg \sigma_{\alpha}$ generate a pseudogroup of isometries $\Phi$ of $\sigma$  that we  call the \emph{holonomy pseudogroup of} $\F$; for details about pseudogroups and orbifolds, see Salem \cite[Appendix D]{Molino}. Like in the classical theory of foliations, $(\sigma,\Phi)$ contains all  essential information about the leaf space $M/\F$.


\begin{remark}
As remarked in \cite{Alex7} we can associate a Weyl group $W$ and a reflection subgroup to an infinitesimally polar foliation $\F$ when $M/\F$ is a good orbifold, i.e., $M/\F$ is equivalent to $\Sigma/W$  where $\Sigma$ is
a connected Riemannian manifold and $W$ is a subgroup of isometries of $\Sigma$.
Then $W$ is called the \emph{Weyl group} of $\F$. Let $\{\psi_{\alpha}\}$ be a maximal collection
of isometries of open subsets of $\Sigma$ onto local sections of $\F$ that gives the equivalence (in the sense of Salem mentioned above)
between $\Sigma/W$ and $M/\F$.  Let $r$  be a reflection in a wall of a local section $\sigma_{\alpha}$.
Then $\mathrm{r}:=\psi_{\alpha}^{-1}\circ r\circ\psi_{\alpha}\in W$  is called a \emph{reflection}
 on $\Sigma$ and $\Gamma$ is the group generated by these reflections.
Also, as remarked in \cite{Alex7}, the results on the Weyl group of a polar foliation discussed in Subsection \ref{subsection-weylgroup} can be generalized to infinitesimal polar foliation, when $M/\F$ is equal to $\Sigma/W.$
\end{remark}

From Theorem \ref{thm-infpol} and the discussion above we infer the next result.

\begin{proposition} \label{prop-global}
Let $\F$ be a closed SRF on a complete Riemannian manifold $M$. The leaf space $M/\F$ is a Riemannian orbifold if and only if $\F$ is infinitesimally polar.
\end{proposition}

At this point we are able to use the equivalent approaches to infinitesimal polarity in order to produce examples where the quotient is not an orbifold.


\begin{proposition} \label{prop-nao-exemplo}
Consider an isometric and effective $\mathbb{S}^1$-action on $M$. Assume that the fixed point set $M^{\mathbb{S}^1}$ is not empty. Then
the orbit space $M/\mathbb{S}^1$ is a Riemannian orbifold if and only if the codimension of $M^{\mathbb{S}^1}$ is $2$.
\end{proposition}
\begin{proof}
To see this, consider the foliation $\F$ given by the partition into $\mathbb{S}^1$-orbits. Now consider $\hat{\F}(p)$ for a fixed point $p$. Regular $\hat{\F}(p)$-leaves, i.e., the one-dimensional orbits, lie in spheres centered at the singular leaf $\{0\}$. On the other hand, they also lie in tubes around the singular stratum $T_pM^{\mathbb{S}^1}$. This leads to regular leaves being  in an affine space parallel to $E:=(T_pM^{\mathbb{S}^1})^{\perp}$. One has that $\hat{\F}(p)$ isoparametric if and only if $\hat{\F}(p)|_{E}$ is isoparametric. Let $L$ be a regular leaf in $E$, i.e., $\dim L=1$. Since $\{0\}$ is the only singular leaf in $E$, it follows $\hat{\F}(p)|_{E}$ is isoparametric if and only if $\dim E = 2$. So we have Proposition \ref{prop-nao-exemplo} from Proposition \ref{prop-global}.
\end{proof}

\begin{example}
From Proposition \ref{prop-nao-exemplo}, the orbit space of the standard action of $\mathbb{S}^1$ on $\mathbb{C}^2$ given by $\zeta\cdot (z_1,z_2):= (\zeta z_1,\zeta z_2)$ cannot be an orbifold. 
\end{example}

\section{Desingularizations of SRF}\label{sec:desingularization}

A SRF $\F$ on a complete Riemannian manifold  $M$ admits a \emph{desingularization} if there exists a
 smooth
surjective map $\rho : \widehat{M}\to  M$ with the following properties: $\widehat{M}$ is a smooth Riemannian manifold foliated by a
regular Riemannian foliation $\widehat{\F}$ and the map  $\rho$ sends leaves
of $\widehat{\F}$ to leaves of $\F.$ In this section we review the desingularizations of polar foliations and desingularizations of SRFs in compact manifolds.


\subsection{The desingularization of polar foliations}

In \cite{Toeben} T\"{o}ben proved that that polar foliations admit desingularizations that preserve the transverse geometry, generalizing some results of
Boualem \cite{Boualem}.

\begin{theorem}[\cite{Toeben}]
\label{Th-Toeben}
Let $\F$ be a polar foliation on a complete Riemannian manifold $M$ of codimension $k$. Then
\begin{enumerate}
\item[a)] Set $\widehat{M}:=\{T_{p}\Sigma\in G_k(TM)|\  p\in N, \Sigma$ is a section of $\F$ through $p\}.$ Then $\widehat{M}$ carries  a natural differentiable structure, for which the inclusion into the Grassmann bundle $G_{k}(TM)$ is an immersion. Moreover, $\widehat{M}$ has a natural
 Riemannian/ totally geodesic bifoliation $(\widehat{\F}, \widehat{\F}^{\perp})$, with respect to the pull-back metric.  We have $\widehat{\F}^{\perp}=\{T\Sigma| \ \Sigma$ is a section of $\F\}$.
\item[b)] The footpoint map $\hat{\pi}: (\widehat{M},\widehat{\F})\rightarrow (M, \F)$ is foliated and maps each horizontal leaf of $\widehat{F}^{\perp}$ isometrically to the corresponding section $\Sigma$ of $\F.$
\end{enumerate}

\end{theorem}
\begin{example}
In the case of a polar action the blow-up is easy to describe.
Let $G$ be a Lie group with a proper polar action on $M$, let $\Sigma$ be a section and $N=N_G(\Sigma)=\{g\in G\mid g(\Sigma)=\Sigma\}$ be the normalizer of $\Sigma$. Then there is a $G$-equivariant diffeomorphism $\Phi:G\times_N\Sigma\to\widehat M$ defined by $\Phi([g,x])=dg_x(T_x\Sigma)\in G_k(TM)$.
\end{example}
Using the desingularization of polar foliations,  T\"{o}ben was able to prove that, the sections of $\F$, as well as their regular leaves, have the same universal cover. He also proved that, by the appropriate choice of a section $\Sigma$, the Weyl pseudogroup extends to a group action on $\Sigma$.

Recently Lytchack \cite{Lytchak} generalized Theorem \ref{Th-Toeben} for  infinitesimally polar foliations.
Using Theorem \ref{thm-infpol} he proves the next result.
\begin{theorem}[\cite{Lytchak}]\label{thm:desing}
Let $\F$ be a SRF on a complete Riemannian manifold. Then $\F$ is infinitesimally polar if and only if  $\F$ has a geometric resolution.
\end{theorem}
Here a \emph{geometric resolution} of $(M,\F)$ is a desingularization that
preserves the length of horizontal curves, i.e., if $\hat{\alpha}$ is a  curve orthogonal to the leaves it meets, then $\rho (\hat{\alpha})$ is a curve orthogonal to the leaves it meets with the same length of $\hat{\alpha}$.



\subsection{Applications of desingularization of polar foliations}

In this section we provide some applications of desingularization to Molino's conjecture and to the problem of  existence of simple foliations.

Recall that a differential $k$-form $\omega$ is said to be \emph{basic} if,
for all $X\in \singularF$, we have $i_{X}\omega=0$ and $i_{X}d\omega=0.$

\begin{theorem}[\cite{AlexLytchak}]
\label{theorem-basic-forms}
Let $(M,\F )$ be an infinitesimally polar  foliation. Let $(\widehat{M}, \widehat{\F})$
be its geometric resolution with the resolution map $\rho:\widehat M\to M$.  Then $\rho$ induces an isomorphism between
the algebra of basic forms.
\end{theorem}
\begin{remark}
If the foliation $\F$ is a polar foliation, then it was already proved in \cite{AlexGorodski} that the algebra of basic forms of $M$ relative
to $\F$ is isomorphic
to the algebra of those differential forms on a section $\Sigma$ that are
invariant under the generalized Weyl pseudogroup of $\Sigma$.
\end{remark}
\begin{proof}

Consider the resolution $\rho:(\widehat M, \widehat{\F}) \to (M,\mathcal F)$.
We first prove this result locally. We claim that $\rho^{*}:\Omega_{b}^{k}(U)\rightarrow \Omega_{b}^{k}(\widehat{U})$ is an isomorphism, where $U$ is an isoparametric neighborhood and $\widehat U$ its blow-up. Note that a section $\Sigma$ on $(U,\F)$ and its Weyl group can be naturally identified with a section of $(\widehat U,\widehat \F)$ and its holonomy group. From \cite[Lemma 3.5]{AlexGorodski} the respective restriction maps to the $\Sigma$ induces isomorphisms $\Omega_{b}^{k}(U)\cong \Omega_b^k(\Sigma)^W\cong  \Omega_{b}^{k}(\widehat{U})$. This implies first the injectivity of $\rho^{*}:\Omega_{b}^{k}(M)\rightarrow \Omega_{b}^{k}(\widehat{M})$  and from that the surjectivity.
\end{proof}

\subsubsection{Molino's conjecture}

In \cite{Molino} Molino proved that, if $M$ is compact, the partition by leaf closures of a (regular) Riemannian foliation $\F$ forms a singular Riemannian foliation. He conjectured that if the same result is true for a SRF $\F$. In \cite{Alex4} this conjecture was proved in the case of polar foliations. In particular, it was proved that the restriction of  transversal Killing
fields to a slice are parallel to their restrictions to the nearest minimal stratum with respect to the normal connection of the intersection of the minimal stratum with the slice.

In what follows we recall how the desingularization can be used to provide a proof of Molino's conjecture  for infinitesimal polar foliation.

\begin{theorem}[\cite{AlexLytchak}]
\label{theorem-molino-conjecture}
Let  $(M, \F )$  be an infinitesimally polar foliation on a complete Riemannian manifold. Then the leaf closures of $\F$ form a partition $\bar {\mathcal F}$ of $M$ which is a singular Riemannian foliation.
\end{theorem}

We start by recalling the definition of basic vector fields. A horizontal vector field $X$ of a (regular) Riemannian foliation is called \emph{basic} (or foliated) if it is basic with respect to a local submersion that describes the plaques of the foliation. A horizontal vector field along a leaf $L$ of an infinitesimal polar foliation $\F$ is \emph{basic} if $X$ is tangent to the stratum $\Sigma$ of $L$ and a basic vector field of $\F|_{\Sigma}$

The duality between basic one-forms and basic vector fields of Riemannian foliations and Theorem \ref{theorem-basic-forms} imply the next lemma.

\begin{lemma}
\label{lemma-molino-conjecture}
 Let $\widehat{X}$ be a basic vector field of $\widehat{\F}$. Consider the continuous vector field $X$ on $M$ such that  $X\circ\rho=\rho_{*}(\hat{X}).$ Then $X$ is a smooth basic vector field of the $\F.$
\end{lemma}

We are ready to prove Molino's conjecture.

First note that the closure $\bar{L}$ of each leaf $L$ of $\F$ is a closed submanifold. This can be proved using the fact that if $L$ and its closure $\bar{L}$ are contained in the same stratum and by the structure theorem of Molino  \cite{Molino}.

Now note that since the resolution $\rho$ is proper, so $\rho(\bar{\hat{F}})=\overline{\rho(\hat{F})}$. Also recall that the resolution projects horizontal geodesics of $\hat{\F}$ to horizontal geodesics of $\F$. These two facts imply that the partition $\bar{\F}=\{\bar{L}\}_{L\in\F}$ is Riemannian, i.e., if a horizontal geodesic meets an element of the partition orthogonally, then it is orthogonal to each element of the partition.

Finally let us prove that $\overline {\mathcal F}$ is a singular foliation. Consider a vector $V$ tangent to the closure of a leaf $L_{p}$  at a point $p$. Let $\widehat{V}$ be a vector with foot point $\hat{p}$ tangent to the closure of a leaf $\widehat{L}_{\hat{p}}$ such that $\rho(\widehat{L}_{\hat{p}})=L_{p}$. It follows from the structure theorem of Molino that there exists a basic transversal Killing vector field $\widehat{X}$ such that $\hat{X}(\hat{p})=\hat{V}$ and $\widehat{X}$ is tangent to the leaves of the singular Riemannian foliation $\bar{\hat{\F}}$.
It follows from Lemma \ref{lemma-molino-conjecture} that the vector field $X$ so that  $X\circ\rho=\rho_{*}(\hat{X})$ is a smooth basic vector field with $X(p)=V$. Since  $\rho(\overline{\widehat{F}})=\overline{\rho(\widehat{F})}$ we conclude that $X$ is tangent to the leaves of $\overline{\F}$.

\subsubsection{Simple foliations}

A SRF on a complete Riemannian manifold $M$
is called \emph{simple} if its leaves are pre-images of a smooth map $H:M\to\Sigma$
where  $\dim\Sigma$ is the codimension of the regular leaves.
A very simple example is the action of $SU(3)$ on itself by conjugation. In this case, the orbits are pre-image of
$ \mathrm{tr}:SU(3)\to \mathbb{C}.$

The question, when a polar foliation is simple, is an old and natural question and has always been raised in the development of the theory of isoparametric foliations and its generalizations; see the surveys of Thorbergsson \cite{Th,ThSurvey2,ThSurvey3}

In \cite{Alex7} it was proved that if the ambient space is simply connected, then the every polar foliation is simple. This result generalizes a result due to the first author and Gorodski \cite{AlexGorodski} for the case of polar foliations with flat sections
and previous results of Heintze \emph{et al}.~\cite{HOL}, Carter and West~\cite{CarterWest2},
and Terng~\cite{terng} for isoparametric submanifolds.  It can also be viewed as a converse to the main results
in~\cite{Alex1} and in Wang~\cite{Wang}.

In this section we explain how the discussion about infinitesimal foliation presented in the last sections and some results of \cite{Alex7,Alex8} can be used to prove that each SRF on a simply connected space is a simple foliation if  its leaf space is a good orbifold.

We start by recalling one of the main results in \cite{Alex7}.

\begin{theorem}[\cite{Alex7}]
\label{proposition-reflectiongroups}
 Let $N$ be a simply connected complete Riemannian manifold, let $\Gamma$ be a reflection
 group on $N$ and let $C\subset N$ be a  fixed chamber, i.e., a fundamental domain for the action of $\Gamma$ on $N.$
  Let $\pi:N\to N/\Gamma=C$ denote the canonical projection.  Then there is a smooth $\Gamma$-equivariant map $G:N \to N$ with the following properties:
 \begin{enumerate}
 \item[(a)] $G|_{C}:C\to C$ is a homeomorphism.
 \item[(b)] Let $K$ be a connected component of an orbit type manifold for the action of $\Gamma$ on $N$.  Then
 the restriction  $G|_{K}$ is a diffeomorphism.
 \item[(c)] The composition $G \circ \pi :N \to N$ is smooth.
 \end{enumerate}
\end{theorem}
\begin{example}
Let us illustrate the above result with a very simple example. Set $N=\mathbb{R}$ and $\Gamma=\{id, -id\}$. Then $C:=[0,\infty)$ is a chamber  and the canonical projection $\pi:N\to N/\Gamma=C$ is given by $\pi(x)=|x|.$ Consider a smooth function $h:[0,\infty)\to [0,\infty)$ such that:
\begin{enumerate}
\item  $h^{(n)}(0)=0$, for all $n$ (where $h^{(n)}$ is the $n$-th derivative of $h$),
\item $h(t)=t$ for $t\geq 1,$
\item $h'(t)\neq 0$ for $t\neq 0$.
\end{enumerate}
 Define the map $G:N\to N$  as $G(0)=0$ and $G(x)=h(|x|)\frac{x}{|x|}$ for $x\neq 0.$ It is not difficult to check that the map $G$ satisfies items (a), (b) and (c).
\end{example}
\begin{proof}

The proof of Theorem \ref{proposition-reflectiongroups} demands several considerations that are too technical to include in this survey. But we can still give an idea of the proof.

First, a new metric $\tilde{g}$ is defined, so that normal slices of the strata are flat and totally geodesic.

The  desired map $G$ can be constructed as a composition of $\Gamma$-equivariant maps $F_i$.
The restriction of each map $F_i$ to  a neighborhood of a stratum $K_i$ is a contraction constructed using the exponential map with respect to the metric $\tilde{g}$.

For example, consider the particular case where  $N=\mathbb{R}^2$, $\Gamma$ is a finite group of reflections in some lines through the origin, and let $K_{1}$ be the union of these lines and  $K_{0}:=\{(0,0)\}$. In this simple case, $\tilde{g}=g$. We define in a neighborhood of $K_1,$
$$F_{1}(x+t v)=x+l_{1}(x)h(\frac{t}{l_{1}(x)})v,$$
where $v$ is a unit vector normal to the lines with footpoint $x$,  $h$ is a function defined above with the additional property that $h^{(n)}$ is a nonnegative increasing function,($t>0$ close to $0$) and $l_1$ is an appropriate function. In a neighboorhood of $K_0$ we have that
$$F_{0}(x,y)=\frac{h(\sqrt{x^{2}+y^{2}})}{\sqrt{x^{2}+y^{2}}}(x,y) ,$$
for $(x,y)\neq (0,0)$ and $F_{0}(0,0)=(0,0).$ In this particular case $G=F_{0}\circ F_{1}$.

Finally, returning to the general case,  the properties of the  metric $\tilde{g}$ allow  one  to verify that $G$ is smooth in a similar way as one can do in Euclidean case.

\end{proof}

\begin{theorem}
Let $\F$ be a SRF on a simply connected complete Riemannian manifold~$M$. Assume that $M/\F$ is a good orbifold $\Sigma/W$.
Then   we can find   a smooth map $H:M\to \Sigma$ such that the
leaves of $\F$ coincide with the level sets of the map $H$ and such that $H(M)$ is isometric to $M/\F$.
\end{theorem}
\begin{proof}
The assumption implies that $(M,\F)$ is infinitesimally polar; see Proposition \ref{prop-global}. Since $M$ is simply connected it follows from \cite{Alex8} that $M/\F=\Sigma/W=\widetilde{\Sigma}/\Gamma$ where $\widetilde{\Sigma}$ is the covering of $\Sigma$ and $\Gamma$ is a reflection (possible empty) group of $\widetilde{\Sigma}$.

Now consider the (infinitesimally polar) desingularization $\rho:(\widehat{M},\widehat{\F})\to (M,\F)$ from Theorem \ref{thm:desing}.
From above we have that $\widehat{M}/\widehat{\F}=\widetilde{\Sigma}/\Gamma$. We must first construct a smooth $\widehat{\F}$-invariant map $\widehat{H}:\widehat{M}\to \Sigma$ that describes the foliation $\widehat{\F}$.

Consider a covering of $\widehat{M}$ by trivial neighborhoods. Let $x$ be a point contained in a trivial neighborhood and take a $\sigma_1$ local section through  $x$.

The assumption that $M/\F$ is equivalent (in the sense of pseudogroups \cite[Appendix D]{Molino}) to $\widetilde{\Sigma}/\Gamma$ implies that there exists a isometry $\psi_1:U_1\subset \widetilde{\Sigma} \to \sigma_1$ and we can define $$\widehat{H}(x)= \pi_{\Sigma}\circ G\circ \pi\circ\psi_{1}^{-1}(x),$$
where $\pi_{\Sigma}:\widetilde{\Sigma}\to \Sigma$ is the universal Riemannian covering map and $G$ is  the map defined in Theorem \ref{proposition-reflectiongroups}.

 We must check that this construction is independent of the choice of trivial neighborhoods and local sections.
Let $y$ be point  contained in the leaf $\widehat{L}_x$,  $\varphi$ a holonomy map such that $\varphi(x)=y$ and consider an isometry $\psi_{2}:U_{2}\to \sigma_{2}$ where $\sigma_{2}$ is a local section that contains $y$. Since $\widehat{M}/\widehat{\F}=\widetilde{\Sigma}/\Gamma$, there exists $\gamma\in \Gamma$ such that
$\gamma:U_{2}\to U_1$ and $\psi_{2}^{-1}\circ\varphi=\gamma^{-1}\circ\psi_{1}^{-1}.$ From this we conclude that
\begin{eqnarray*}
\pi_{\Sigma}\circ G\circ\pi\circ\psi_{2}^{-1}(y)&=&\pi_{\Sigma}\circ G\circ\pi\circ\psi_{2}^{-1}\circ\varphi(x)\\
  &=& \pi_{\Sigma}\circ G \circ\pi\circ \gamma^{-1}\circ\psi_{1}^{-1}(x)\\
  &=&\pi_{\Sigma}\circ G\circ \pi \circ \psi_{1}^{-1}(x)
\end{eqnarray*}
and hence $\widehat{H}$ is well defined. $\widehat{H}$ is smooth since $G\circ\pi$ is  smooth; recall Theorem \ref{proposition-reflectiongroups}.

Finally consider the map $H: M\to \Sigma$ such that $H\circ\rho =\widehat{H}$. From Theorem \ref{theorem-basic-forms} it follows that $H$ is smooth.

\end{proof}

In \cite{Lytchak} Lytchak proved that the leaf space of a variationally complete foliation on a simply connected Riemannian manifold $M$ is a good orbifold. Therefore we can conclude from the above result the next corollary.

\begin{corollary}
Polar foliations and variationally complete foliations on simply connected Riemannian manifolds are simple.
\end{corollary}

\subsection{The general desingularization of SRF}

As observed by Lytchak \cite{Lytchak}, if a SRF $\F$ is not infinitesimally polar, then it does not admit a desingularization that preserves the transverse geometry.
Nevertheless, as we will see in this section, even in the general case a SRF $\F$ admits a desingularization that preserves ``part'' of the transverse geometry.
This fact  will allow us to prove that $M/\F$ is a  Gromov-Hausdorff limit of orbifolds $\widehat{M}_{n}/\widehat{\F}_{n}$, when $M$  and the leaves are compact.

\begin{theorem}[\cite{Alex6}]
\label{cor-epsilon-isometria}
Let $\F$ be a SRF  with compact leaves on a compact Riemannian manifold $(M,g)$.
 Then for each small positive number $\epsilon$ we can find a regular Riemannian foliation $\widehat{\F}_{\epsilon}$ on a compact Riemannian manifold $\widehat{M}_{\epsilon}$ and a desingularization map $\hat{\rho}_{\epsilon}:\widehat{M}_{\epsilon}\rightarrow M$  with the following property:
if $\hat{q},\hat{p}\in \widehat{M}_{\epsilon}$ then
\[ |d(L_{q},L_{p})-\hat{d}(\hat{L}_{\hat{q}},\hat{L}_{\hat{p}})|<\epsilon\]
where $p=\hat{\rho}_{\epsilon}(\hat{p})$ and $q=\hat{\rho}_{\epsilon}(\hat{q})$.
\end{theorem}

The above result implies  that $M/\F$ is a limit of a sequence of Riemannian orbifolds  with respect to the Gromov-Haudorff distance, a concept that we now recall.

Let $X$ be a metric space and $A, B\subset X$. Then we define the \emph{Hausdorff distance} between $A$ and $B$ as  $d_{H}(A,B)=\inf\{ \epsilon\mid A\subset \tub_{\epsilon}(B), B\subset\tub_{\epsilon}(A)\}$. For two metric spaces $X$ and $Y$ an \emph{admissible metric} on the disjoint union $X\sqcup Y$
is a metric that extends the given metrics on $X$ and $Y$. With this we can define the \emph{Gromov-Hausdorff distance} as
$$ d_{G-H}(X,Y):=\inf\{d_{H}(X,Y)\mid \mathrm{admissible\, metrics\, on\, }  X\sqcup Y\} .$$
It is possible to prove that if $X$ and $Y$ are compact metric spaces then $X$ and $Y$ are isometric if and only if $d_{G-H}(X,Y)=0$.
Since $d_{G-H}$ is symmetric and satisfies the triangle inequality, the collection of compact metric spaces $(\mathcal{M},d_{G-H})$ turns out to be a pseudometric space, and if we pass to equivalence classes of isometric spaces it becomes a metric space. In fact, it is possible to prove that this metric space is complete and separable; see Petersen \cite{petersen}.

The  last theorem and basic results about Gromov-Hausdorff distance (see \cite[p. 295]{petersen}) imply the next corollary.

\begin{corollary}[\cite{Alex6}]
\label{cor-gromov-hausdorff-convergence}
Let $\F$ be a SRF  with compact leaves on a compact Riemannian manifold $(M,g)$.
 Then for each small positive number $\epsilon$ we can find a regular Riemannian foliation $\hat{\F}$ with compact leaves on a compact Riemannian manifold $\hat{M}$ so that
$d_{G-H}(M/\F, \hat{M}/\hat{\F})<\epsilon$. In particular, $M/\F$ is a Gromov-Hausdorff limit of a sequence of Riemannian orbifolds $\{(\hat{M}_{n}/\hat{\F}_{n})\}.$
\end{corollary}

\begin{remark}
Theorem \ref{cor-epsilon-isometria} provides more information than the above corollary, because, apart from the existence of a 1-parameter family of orbifolds
$\widehat{M}_{\epsilon}/\widehat{\F}_{\epsilon}$ that converges to the Alexandrov space $M/\F$, it also guarantees the existence of a 1-parameter family of continuous maps,
$\hat{\rho}_{\epsilon}:\widehat{M}_{\epsilon}/\widehat{\F}_{\epsilon} \to M/\F$,  that turns out to be \emph{smooth in the sense of Schwarz} \cite{Schwarz}, i.e., the pull-back of smooth basic functions are smooth basic functions. This suggests that  Theorem \ref{cor-epsilon-isometria} could be useful  in  studying the smooth structure of $M/\F$. For more remarks about the smooth structure of $M/\F$ see \cite{AlexLytchak,AlexRadeschi}.
\end{remark}

Before we sketch the proof of Theorem \ref{cor-epsilon-isometria} we need to look closer at the desingularization map $\rho$ that turns out to be a composition of  blow-ups along the strata.

We start by recalling the definition of a blow-up along a submanifold.
\begin{lemma}
\label{lemma-projectiveblow-usp}
Let $\tub_{r}(\minimalstratum)$ be a small neighborhood of the minimal stratum $\minimalstratum$. Then
\begin{enumerate}
\item[(a)] $\hat{N}:=\{(x,[\xi])\in \tub_{r}(\minimalstratum)\times \mathbb{P}(\nu \minimalstratum)| x= \exp^{\perp}(t\xi)\ \mbox{for some}\ |t|<r\}$ is a smooth manifold (called blow-up  of $\tub_{r}(\minimalstratum)$  along $\minimalstratum$) and $\hat{\pi}:\hat{N}\rightarrow \tub_{r}(\minimalstratum)$, defined as $\hat{\pi}(x,[\xi])=x$ is also smooth.
\item[(b)]$\hat{\minimalstratum}:=\hat{\pi}^{-1}(\minimalstratum)=\{(\rho([\xi]),[\xi])\in \hat{N}\}=\mathbb{P}(\nu\minimalstratum)$, where $\rho: \mathbb{P}(\nu \minimalstratum)\rightarrow \minimalstratum$ is the canonical projection.
\item[(c)] There exists a  singular foliation $\hat{\F}$ on $\hat{N}$ so that  $\hat{\pi}: (\hat{N}-\hat{\minimalstratum}, \hat{\F})\rightarrow (\tub_{r}(\minimalstratum)-\minimalstratum, \F)$ is a foliated diffeomorphism.
\end{enumerate}
\end{lemma}

Now we can discuss how the desingularization map is constructed and understand some of its properties that will be used in the proof of Thereom \ref{cor-epsilon-isometria}. Note that in the result below we do not need to assume that the leaves are compact.

\begin{theorem}[\cite{Alex6}]
\label{thm-Blowup-srf}
Let $\F$ be a SRF of a compact Riemannian manifold $(M,g)$, $\minimalstratum$ the minimal stratum of $\F$ (with leaves of dimension $k_{0}$) and $\tub_{r}(\minimalstratum)$ the distance tube of $\minimalstratum$ of radius $r$.
Then, by blowing up $M$ along $\minimalstratum,$ we have a SRF $\hat{\F}_{r}$ (with leaves of dimension greater then $k_{0}$) on a compact Riemannian manifold $(\hat{M}_{r}(\minimalstratum),\hat{g}_{r})$  and a  map $\rho_{r}:\hat{M}_{r}(\minimalstratum)\rightarrow M$ with the following properties:
\begin{enumerate}
\item[(a)] $\rho_{r}$ projects each leaf of $\hat{\F}_{r}$ onto a leaf of $\F$.
\item[(b)]Set $\hat{\minimalstratum}:=\rho_{r}^{-1}(\minimalstratum)$. Then  $\rho_{r}:(\hat{M}_{r}(\minimalstratum)-\hat{\minimalstratum},\hat{\F}_{r})\rightarrow (M-\minimalstratum,\F)$ is a foliated diffeomorphism and $\rho_{r}:\hat{M}_{r}(\minimalstratum)-\tub_{r}(\hat{\minimalstratum})\rightarrow M-\tub_{r}(\minimalstratum)$ is an isometry
\end{enumerate}
\begin{enumerate}
\item[(c)]If a unit speed geodesic $\hat{\gamma}$ is orthogonal to $\hat{\minimalstratum}$, then $\rho_{r}(\hat{\gamma})$ is a unit speed geodesic orthogonal to $\minimalstratum$.
\item[(d)]$\rho_{r}|_{\hat{\minimalstratum}}: (\hat{\minimalstratum},\hat{g}_{r})\rightarrow(\minimalstratum,g)$ is a Riemannian submersion. In addition $(\hat{\minimalstratum},\hat{\F}_{r}|_{\hat{\minimalstratum}},\hat{g}_{r})$ is a SRF and the lifts of horizontal geodesics of $(\minimalstratum,\F|_{\minimalstratum},g)$ are horizontal geodesics of $(\hat{\minimalstratum},\hat{\F}_{r}|_{\hat{\minimalstratum}},\hat{g}_{r}).$
\end{enumerate}
 Furthermore, by successive blow-ups of the minimal strata, we obtain a regular Riemannian foliation $\widehat{\F}$ on a compact Riemannian manifold $\widehat{M}$ and a desingularization map $\hat{\rho}:\widehat{M}\rightarrow M$ that projects each leaf $\widehat{L}$ of $\widehat{\F}$ onto a leaf $L$ of $\F.$
\end{theorem}
\begin{proof}

We will briefly sketch the main ideas behind the proof.
The construction of the metric $\hat{g}_{r}$ on the blow-up space $\hat{M}_{r}(\minimalstratum)$ will require several steps.

The first step is the construction of a metric $\tilde{g}$ and a distribution $H$ on a neighborhood of $\minimalstratum$ with special properties.

\begin{lemma}
\label{lemma1-thm-Blowup-srf}
There exists  a neighborhood  $\tub_{r}(\minimalstratum)$ of $\minimalstratum$   and a metric  $\tilde{g}$ on $\tub_{r}(\minimalstratum)$   such that:
\begin{enumerate}
\item[(a)] $\F|\tub_{r}(\minimalstratum)$ is a Riemannian foliation with respect to $\tilde{g}$
and the associated transverse distance is not changed.
\item[(b)] If a curve $\gamma$ is a unit speed geodesic segment orthogonal to $\minimalstratum$ with respect to the original metric $g$, then
$\gamma$ is a unit speed geodesic segment orthogonal to $\minimalstratum$ with respect to the new metric $\tilde{g}$.
\item[(c)] There exists a smooth distribution $H$, whose dimension is equal to the codimension of the leaves in $\minimalstratum$, such that the normal space  of each plaque of $\F|_{\tub_{r}(\minimalstratum)}$ (with respect to $\tilde{g}$) is contained in $H.$
\item[(d)] The restriction of the metric $\tilde{g}$ to the stratum $\minimalstratum$ coincides with the restriction of the original metric $g$ to $\minimalstratum.$
\end{enumerate}
\end{lemma}
We remark that the distribution of the above lemma is defined as $H_{\exp(\xi)}:=T_{\exp(\xi)}S_{x}$ where $\xi\in \nu_{x}\minimalstratum$ and $S_{x}$ is a slice of $L_{x}$  at $x$  with respect to the original metric $g$.

We now come to the second step of our construction, in which we change  the metric $\tilde{g}$ in
some directions, getting a new metric $\hat{g}^{M}$ on $\tub_{r}(\minimalstratum)-\minimalstratum$.

First note that, for small $\xi\in \nu_{q}\minimalstratum,$ we can decompose $T_{\exp_{q}(\xi)}M$ as a direct sum of orthogonal subspaces (with respect to the metric $\tilde{g}$) 

\[T_{\exp_{q}(\xi)}M= H_{\exp_{q}(\xi)}^{\perp}\oplus H_{\exp_{q}(\xi)}^{1}\oplus H_{\exp_{q}(\xi)}^{2}\oplus H_{\exp_{q}(\xi)}^{3},\]
where $H_{\exp_{q}(\xi)}^{\perp}$ is orthogonal to $H_{\exp_{q}(\xi)}$ and  $H_{\exp_{q}(\xi)}^{i}\subset H_{\exp_{q}(\xi)}$ is defined below.
\begin{enumerate}
\item $H_{\exp_{q}(\xi)}^{1}$ is the line generated by $\frac{d}{d t}\exp_{q}(t\xi)|_{t=1}$.
\item Set $T_{q}=\exp_{q}(\nu(\minimalstratum)\cap B_{\epsilon}(0))$. Then $H_{\exp_{q}(\xi)}^{2}$ is the orthogonal complement of  $H_{\exp_{q}(\xi)}^{1}$ in $T_{\exp_{q}(\xi)}T_{q}$.
\item $H_{\exp_{q}(\xi)}^{3}$ is the orthogonal complement of $T_{\exp_{q}(\xi)}T_{q}$ in $H_{\exp_{q}(\xi)}$.
\end{enumerate}

Now we define a new metric $\hat{g}^{M}$ on
$\tub_{r}(\minimalstratum)-\minimalstratum$ as follows:

\begin{equation}
\label{def-metrica-hatgM}
\hat{g}^{M}_{\exp_{q}(\xi)}(V,W):= \tilde{g}(V_{\perp},W_{\perp})+ \tilde{g}(V_{1},W_{1})
                                +\frac{r^{2}}{\|\xi\|^{2}} \tilde{g}(V_{2},W_{2})+  \tilde{g}(V_{3},W_{3}),
\end{equation}
where $V_{i},W_{i}\in H_{\exp_{q}(\xi)}^{i}$ and  $V_{\perp},W_{\perp}\in H_{\exp_{q}(\xi)}^{\perp}.$

One can use Lemma \ref{lemma1-thm-Blowup-srf}, Proposition \ref{prop-metric-stratum} and Proposition \ref{slidingalongStratum}  to prove

\begin{lemma}
$\F$ is a SRF on $\tub_{r}(\minimalstratum)-\minimalstratum$ with respect to $\hat{g}^{M}.$ In addition, if $\gamma:[0,a]\rightarrow \tub_{r}(\minimalstratum)$ is a unit speed geodesic orthogonal to $\minimalstratum$ with respect to the original metric $g$, then $\gamma|_{(0,a]}$ is a unit speed  geodesic with respect to $\hat{g}^{M}.$
\end{lemma}

In the third step of our construction, we pull back the metric $\hat{g}^M$ to the blow-up of $\tub_{r}(\minimalstratum)$ along $\minimalstratum$, denoted here by $\hat{N}$ (recall the notation of Lemma \ref{lemma-projectiveblow-usp}), and then prove that the induced foliation $\hat{\F}$ on $\hat{N}$ is a SRF with respect to this new metric.

Considering the pullback metric $\hat{g}:=\hat{\pi}^{*}\hat{g}^{M}$ and
using the Remark \ref{rem-prop-metric-stratum} and Proposition \ref{prop-metric-stratum} one can prove

\begin{lemma}
The metric $\hat{g}$ can be extended to $\hat{N}$  such that the following properties hold:
\begin{enumerate}
\item[(a)] if a unit speed geodesic $\hat{\gamma}$ is orthogonal to $\hat{\minimalstratum}$ with respect to $\hat{g}$, then $\hat{\pi}(\hat{\gamma})$ is a unit speed  geodesic orthogonal to $\minimalstratum$ with respect to the original metric $g$.
\item[(b)] $\hat{\pi}|_{\hat{\minimalstratum}}: (\hat{\minimalstratum},\hat{g})\rightarrow(\minimalstratum,g)$ is a Riemannian submersion.
\item[(c)] $(\hat{N},\hat{\F},\hat{g})$ is a SRF.
\item[(d)]$(\hat{\minimalstratum},\hat{\F}|_{\minimalstratum},\hat{g})$ is a SRF
and the lifts of horizontal geodesics of $(\minimalstratum,\F|_{\minimalstratum},g)$ are horizontal geodesics of $(\hat{\minimalstratum},\hat{\F}|_{\minimalstratum},\hat{g})$.
\end{enumerate}
\end{lemma}

In the last step of our construction, we   glue $\hat{N}$ with a copy of $M-\tub_{r}(\minimalstratum)$ and  construct the space
$\hat{M}_{r}(\minimalstratum)$ and  the projection $\rho_{r}:\hat{M}_{r}(\minimalstratum)\rightarrow M$.
A natural singular foliation  $\hat{\F}_{r}$ is induced on $\hat{M}_{r}(\minimalstratum)$ in analogy to the blow-up of isometric actions, see e.g. \cite{Duistermaat}.
To define the appropriate metric   $\hat{g}_{r}$ on  $\hat{M}_{r}(\minimalstratum)$ consider  a partition of unity of $\hat{M}_{r}(\minimalstratum)$ by two functions $\hat{f}$ and $\hat{h}$ such that
 \begin{enumerate}
 \item[(a)] $\hat{f}=1$ in $\tub_{r/2}(\hat{\minimalstratum})$ and $\hat{f}=0$ outside of $\tub_{r}(\hat{\minimalstratum})$.
 \item[(b)] $\hat{f}$ and $\hat{h}$ are constant on the cylinders  $\partial \tub_{\epsilon}(\hat{\minimalstratum})$ for $ \epsilon<2 r.$
 \end{enumerate}
 Set $\hat{g}_{r}:=\hat{f}\hat{g}+\hat{h} g$ and use the last lemma and Proposition \ref{prop-metric-stratum} to conclude the proof.

\end{proof}


We are finally ready to sketch the proof of Theorem \ref{cor-epsilon-isometria} that will follow direct from the next lemma.
\begin{lemma}
Let  $\minimalstratum$ be the minimal stratum of $\F$. Consider the Riemannian manifold  $(\hat{M}_{r}(\minimalstratum),\hat{g}_{r})$   and  the foliation $\hat{\F}_{r}$ defined in Theorem  \ref{thm-Blowup-srf}. Then for each small $\epsilon>0$ there exists $r$ such that
 for each $\hat{q}, \hat{p}\in \hat{M}_{r}(\minimalstratum)$ we have
\[ |d(L_{q},L_{p})-\hat{d}_{r}(\hat{L}_{\hat{q}},\hat{L}_{\hat{p}})|<\epsilon\]
where $p:=\rho_{r}(\hat{p})$ and $q:=\rho_{r}(\hat{q}).$
\end{lemma}
\begin{proof}
Due to (c) and (d) of Theorem \ref{thm-Blowup-srf} and equation \eqref{def-metrica-hatgM},
for a small $\epsilon$ there exists $r>0$ with the following properties:
if $\hat{\gamma}$ is a minimal horizontal geodesic that joins $\hat{L}_{\hat{p}}$ to $\hat{L}_{\hat{q}}$ there exists a horizontal  curve
$\hat{\gamma}^{\epsilon}$ in $\hat{M}_{r}(\minimalstratum)$ so that
\begin{enumerate}
\item $l(\rho_{r}(\hat{\gamma}^{\epsilon}))\leq \hat{l}_{r}(\hat{\gamma}^{\epsilon}),$
\item $|\hat{l}_{r}(\hat{\gamma}^{\epsilon})-\hat{l}_{r}(\hat{\gamma})|<\epsilon.$
\end{enumerate}

Therefore
\begin{eqnarray*}
d(L_{p},L_{q})&\leq & l(\rho_{r}(\hat{\gamma}^{\epsilon}))\\
              &\leq & \hat{l}_{r}(\hat{\gamma}^{\epsilon})\\
              &< & \hat{l}_{r}(\hat{\gamma})+\epsilon\\
              &= & \hat{d}_{r}(\hat{L}_{\hat{p}},\hat{L}_{\hat{q}})+\epsilon.
\end{eqnarray*}

In an analogous way, because of items (c) and (d) of Theorem \ref{thm-Blowup-srf} and equation \eqref{def-metrica-hatgM}, if
 $\gamma$ is a minimal horizontal geodesic that joins $L_{p}$ to $L_{q}$ there exists a curve  $\gamma^{\epsilon}$ in $M$ and a  curve
$\hat{\gamma}$ in $\hat{M}_{r}(\minimalstratum)$ that joins $\hat{L}_{\hat{p}}$ to $\hat{L}_{\hat{q}}$ such that
\begin{enumerate}
\item $\hat{l}_{r}(\hat{\gamma})\leq l(\gamma^{\epsilon})+\frac{\epsilon}{2}$
\item $|l(\gamma^{\epsilon})- l(\gamma)|<\frac{\epsilon}{2} $
\end{enumerate}
Therefore
\begin{eqnarray*}
\hat{d}_{r}(\hat{L}_{\hat{p}},\hat{L}_{\hat{q}}) & \leq & \hat{l}_{r}(\hat{\gamma})\\
                                                 & \leq & l(\gamma^{\epsilon})+\frac{\epsilon}{2}\\
                                                 &<& l(\gamma)+\epsilon\\
                                                 & = & d(L_{p},L_{q})+\epsilon.
\end{eqnarray*}
\end{proof}




\begin{remark}[errata]
Small corrections should be made in the proofs of \cite{Alex6}. As in  equation \eqref{def-metrica-hatgM}, the  equation (3) of  \cite{Alex6} should have the factor $r^{2}$. This is not important in the proof of Theorem \ref{thm-Blowup-srf} but plays a role in the proof of Theorem \ref{cor-epsilon-isometria}. In the proof of  Theorem \ref{cor-epsilon-isometria} one has
$l(\hat{\pi}_{r}(\hat{\gamma}^{\epsilon}))\leq \hat{l}_{r}(\hat{\gamma}^{\epsilon})$ (instead $l(\hat{\pi}_{r}(\hat{\gamma}^{\epsilon}))= \hat{l}_{r}(\hat{\gamma}^{\epsilon})$), and $\hat{l}_{r}(\hat{\gamma})\leq l(\gamma^{\epsilon})+\frac{\epsilon}{2}$
(instead $\hat{l}_{r}(\hat{\gamma})= l(\gamma^{\epsilon})$).  Apart from these misprints the proofs of \cite{Alex6} are correct.

\end{remark}



\section{Perspectives}
\label{sec:question}

There are many other interesting results we did not cover. Here we want to point out some of them.

A polar foliation on the sphere or in Euclidean space is the decomposition by parallel submanifolds of an isoparametric submanifolds. A polar representation is orbit equivalent to the isotropy representation of a symmetric space by Dadok \cite{Dadok}. Thorbergsson showed in \cite{Th91} that any compact, full, irreducible polar foliation of Euclidean space (i.e. an isoparametric foliation) of codimension larger 2 is the orbit foliation of a polar representation. 

There is a similar development for symmetric spaces of compact type. The classification of hyperpolar actions on irreducible symmetric spaces of compact type was achieved by Kollross. Recently, Kollross and Lytchak showed in \cite{KollrossLytchak} that on irreducible compact symmetric spaces of rank greater than one every nontrivial polar action is already hyperpolar. Polar actions on compact rank one symmetric spaces were classified by Podest\`a and Thorbergsson in \cite{ThorbPodesta}.


A closed polar foliation on an Hadamard manifold is diffeomorphic to the product of a compact isoparametric foliation on Euclidean space and the trivial foliation on Euclidean space with one leaf due to the third author \cite{Toeben2}. In the homogeneous case we have the following. A proper polar action of a Lie group $G$ on an Hadamard manifold is equivariantly diffeomorphic to a homogeneous bundle $G\times_K \RR^n$ over $G/K$ where $K$ is the maximal compact subgroup of $G$ and the action of $K$ on $\RR^n$ is an isotropy representation of a symmetric space, see \cite{Toeben2}.
Progress in the classification of regular hyperpolar actions on symmetric spaces of noncompact type up to isometry has been made by Berndt, Tamaru, Diaz-Ramos and others, see \cite{Berndt1,Berndt2}. 

By a work of Lytchak \cite{Lytchak09} no SRF foliation exists on a compact negatively curved manifolds (for the regular case, see Zeghib \cite{Zeghib},  for the polar case, see \cite{Toeben2} by the third author).

An important question is the classification of SRFs in the Euclidean space or the sphere. In spheres Riemannian foliations were classified by Grove-Gromoll \cite{GromollGrove} in dimension one to three, and SRFs of dimension one to three by Radeschi \cite{Radeschi}.

As remarked in Gromoll-Walshap \cite{Gromoll-Walshap} there has been increasing realization that Riemannian foliations play a key role in understanding the structure of manifolds with positive or nonnegative curvature. Wilking used the dual foliation of a SRF to show that the metric projection onto the soul is smooth.
Another result by Wilking is that there is only one dual leaf in a manifold of positive curvature, in other words, that every two points can be connected by a piece-wise smooth horizontal curve; see \cite{WilkGAFA}.

Another topic is the smooth structure of the leaf space $M/\F$.
Some natural questions and directions are pointed out in \cite{AlexLytchak}. In particular, an analogue of the Meyers-Steenrod theorem for the leaf space is proved in \cite{AlexRadeschi}. When the foliation is homogeneous, i.e., the leaf space is an orbit space, very important results have been obtained by Schwarz, see e.g. \cite{Schwarz,Schwarz2}. We hope that the desingularization of SRFs turns out to be a useful tool in the study of certain questions about the smooth structure of $M/\F$ as well as some of their topological invariants, as basic cohomology.

Finally we would like to cite the  work of  Wiesendorf \cite{Wiesendorf} where
SRF with Taut leaves are studied. In particular he shows that tautness 
of a singular Riemannian foliation  is actually a property of the quotient.


\bibliographystyle{amsplain}

\end{document}